\documentclass{amsart}

\usepackage{amssymb}

\newtheorem{theorem}{Theorem}[section]
\newtheorem*{Lemma4A}{Lemma 4A}
\newtheorem*{Lemma4B}{Lemma 4B}

\theoremstyle{definition}

\theoremstyle{remark}

\numberwithin{equation}{section}

\DeclareMathOperator*{\lcm}{lcm}

\begin{document}

\title[Carmichael relations: three-prime Carmichael numbers up to $10^{24}$]{Carmichael number variable relations:\\ three-prime Carmichael numbers up to $10^{24}$}

\author{J.M.Chick}
\address{10 Postwood Green, Hertford Heath, Hertfordshire  SG13 7QJ, UK}




\begin{abstract}
Bounds and other relations involving variables connected with Carmichael numbers are reviewed and extended.
Families of numbers or individual numbers attaining or approaching these bounds are given. A new algorithm for finding three-prime Carmichael numbers is described, with its implementation 
up to $10^{24}$.
Statistics relevant to the distribution of three-prime Carmichael numbers are given, with particular reference
to the conjecture of Granville and Pomerance in \cite{Granville}.
\end{abstract}

\maketitle

\section{Introduction}\label{Sec1}
A Carmichael number $n$ is defined by the property that $n$ is composite and ${a^n \equiv a \pmod n}$ for all $a$.
For $n = \prod^d_{i=1}{p_i}^{\alpha_i}$ , with $p_i$ prime, Korselt in~1899~\cite{Korselt} stated that $\alpha_i = 1$
for all $i$ and $\lcm(p_1-1, p_2-1,\dotsc ,p_d-1)$ divides $(n-1)$ is a necessary and sufficient condition for
$n$ to divide $(a^n-a)$, but he did not exhibit any such number $n$. 
In~1910 Carmichael~\cite{Carmichael2} showed that the above condition required $d \ge 3$ and all $p_i$ to be odd, 
and gave four such numbers, the smallest of which was $561 = 3 \cdot 11 \cdot 17.$ 
In~1912~\cite{Carmichael3} he amplified his remarks and extended his list to fifteen such numbers,
including one with $d=4$ (although very curiously he reconsidered and rejected 561!)

Korselt's criterion, stated above, is the basis for much of the theory on Carmichael numbers and for algorithms to find them,
including ours. For a background on Carmichael numbers and previous counts of Carmichaels
 up to increasing upper bounds see Ribenboim~\cite{Ribenboim4}, counts which have now culminated in Richard~Pinch's 
 up to $10^{18}$ \cite{Pinch10}. 
Our list up to $10^{24}$ for $d=3$
 may be found on the website of the Cambridge University Department of Pure Mathematics and Mathematical Statistics~\cite{Camftp}.

In addition to $p\,_i$ and $n$ already mentioned, Korselt's criterion spawns numerous other variables, some of them specific to $d=3$,
 and various relationships and bounds connecting them are of value in constructing algorithms to find Carmichael numbers as well as 
being of interest in themselves. 
In the next three sections we review and extend such relationships and bounds.
\section{Notation; variables; Korselt factorisations, numbers and families}\label{Sec2}

\subsection{Notation, KN's and K-families}\label{Sec2a}
Because of the form of the Korselt criterion, we shall find it convenient 
consistently and exclusively to use the abbreviation $x': = x-1$.
So we have $(xy)' = xy'+x'=x'y'+x'+y'$, etc. We shall also consistently use the notation which we define during \S\ref{Sec2},
 without repeated explanation. 
 
Let $n= \prod^d_{i=1} p_i$, where $1<p_1<p_2<\ldots<p_d$ and $d \ge3$, be a number for which the factors $p_i$ satisfy the condition
 $p_i{}'$ divides $n'$, and define $P:=\prod^{d-2}_{i=1}p_i$, $p:=p_1$, $q:=p_{d-1}$ and $r:=p_d$,
 so for $d=3$, $n=pqr$. Let $P_i:=n/{p_i}$,
 so $n=p_iP_i$ for $1 \le i\le d$, and similarly write $n=qQ=rR$.

Then $n'=(p_iP_i)'=p_i{}'P_i+P_i{}'$, so $n'\equiv 0\pmod {p_i{}'}$ 
gives $P_i{}' \equiv 0 \pmod {p_i{}'}$ for $1 \le i \le d$, and conversely.
 
Thus there exist integers $\lambda_i$ such that $P_i{}'=\lambda_ip_i{}'$.
Also if $\lambda_d=1$, then obviously $p_d=P_d=\prod^{d-1}_{i=1} p_i$ is {\it necessarily} composite
 (usually we shall assume that $\lambda_d \ge2$).
We write
 $D:=\lambda_{d-1}\text{  ,}\quad E:=\lambda_d\text{ ,}$
 so \begin{equation}Q\,'=Dq\/'\quad\text{and}\quad R\,'=Er'\end{equation}\label{B1}

\vskip-12pt
So far in substance but not in notation we follow Carmichael, if $p_i$ are all odd primes.
But both for algorithms and theoretical results it is often necessary to consider sets of numbers $n$ with a factorisation
 which satisfies the Korselt divisibility criterion without all (or any) of the factors necessarily being prime:
 such a number, $n$ as above, together with the particular Korselt factorisation, we shall term a {\it Korselt number}
 (abbreviated to $KN$, or $K_dN$ if its Korselt factorisation has $d$ factors) if, for all $i$, $p_i$ is odd and $E\ge2$.
I have not established whether it is possible for a number to be a $K_3N$ in more than one way, but the Korselt factorisation will
 always be apparent from the context. Clearly if every $p_i$ is a prime then $n$ is a Carmichael number, 
which we shall abbreviate to $CN$ or $C_dN$ in like manner.
We shall also consider certain families of $KN$s ($K$-families or $K_d$-families) 
of the type $n(t)=\prod^d_{i=1}p_i(t)$, where $n$ and $p_i$
 are polynomials over the integers and $t$ is a non-negative integer parameter. 
It seems plausible to conjecture (with Schinzel,
 see page~91 of~\cite{Ribenboim4}) that any $K$-family will contain an infinite number of $CN$'s unless, speaking loosely,
 there is some {\it obvious} reason why (almost) all members have at least one composite $p_i$.


\subsection{Chernick's universal forms}\label{Sec2b}
The best known $K$-families are the ``universal forms'' described by Chernick in~1939~\cite{Chernick5},
and it will be helpful to summarise his theory in our notation.
Let $n$ be any $KN$.
Then we have 
\begin{equation}n'=\prod^d_{i=1}(p_i{}'+1)-1=\prod^d_{i=1}p_i{}'
+\sum (p_1{}'p\,_2{}'\dotsm p\,_{d-1}{}')+\ldots+\sum p_1{}'p\,_2{}'+\sum p_1{}'
\label{B2}\end{equation}
Let $H=\underset{i\neq j}{\gcd}{[\{p_i{}'\}]}$ for any particular $j$. 
Then from (\ref{B2}) ${n'\equiv p_j{}'\pmod H}$.
But since $n$ is a $KN$, clearly for $i \neq j$ we have $H|p_i{}'|n'$, so ${n'\equiv 0 \pmod H}$.
Hence $p_j{}'\equiv0\pmod H$, whence $H=\underset{1\leq i \leq d}{\gcd}[\{p_i{}'\}]$.
So if $A_i:=p_i{}'/H$, any set of $(d-1)A_i${}'s are coprime, i.e.
\begin{equation}\underset{i\neq j}{\gcd}[\{A_i\}]=1\text{ ,  for}\quad 1\leq j\leq d \label{B3}\end{equation}
$\text{Also if   }L:=\underset{1\leq i\leq d}{\lcm}[\{p_i{}'\}]\quad\text{and}\quad
\ell:=\underset{1\leq i\leq d}{\lcm}[\{A_i\}]$, clearly $L=\ell H$.
Combining $p_i{}'=A_iH$, $L=\ell H$ and $n'\equiv 0\pmod L$ with (\ref{B2}), we get
\begin{equation}\begin{split}H^{d-2}\sum A_1A_2\dotsm A_{d-1}+H^{d-3}\sum A_1A_2\dotsm A_{d-2}\\
+\dotsb+H\sum A_1A_2+\sum A_1\equiv0\pmod\ell\label{B4}\end{split}\end{equation}
Suppose now that we are given any set $\{A_1,A_2,\dotsc,A_d\}$ satisfying (\ref{B3}),
then congruence (\ref{B4}) is always soluble for $H$ when $d=3$ (see below, following (\ref{B5})), 
but not necessarily when $d>3$. 
If $H_o$ is any solution,
then so is $H_t=H_o+t\ell$, so we can choose $H_o$ to satisfy $1\leq H_o\leq\ell$.
Then if we take \hbox{$p_i=A_iH_t+1=A_i\ell t+A_iH_o+1$} and $n=\prod_{i=1}^d p_i$, $n$ satisfies the Korselt criterion for all $t$,
and with certain precautions yields a $K_d$-family corresponding to each basic solution $H_o$ (precautions: 
our definition of a $KN$ requires
(i) all $p_i$ are odd, and (ii) $E\geq2$: for (i), if $\ell$ is even and $H$ is odd, from (\ref{B4}) we get
\begin{equation*}\sum(A_1A_2\dotsm A_{d-1})+\dotsb+\sum A_1A_2+\sum A_1
=\prod_{i=1}^d(A_i+1)-\prod_{i=1}^d A_i-1\equiv0\pmod2\text{ ,}\end{equation*}
whence since $\ell\medspace|\prod_{i=1}^d A_i$ and $\ell$ is even, 
$\prod^d_{i=1}(A_i+1)\equiv1\pmod2$ and so for all $i$, $A_i\equiv0\pmod2$,
contradicting (\ref{B3}) which holds by hypothesis; so if $\ell$ is even then $H$ is even and all $p_i$ are odd; 
but $\ell$ is odd if{}f all $A_i$ are odd, so $H_t$ is
 \hbox{alternately} odd or even as $t$ increases, and then the $K$-family will be generated
by the parameter $u$ where $t=2u$ or $2u+1$ according as $H_o$ is even or odd; while for (ii), 
$E\geq2$, it may be necessary to exclude $t=0$ from the family).
These $K$-families are Chernick's universal forms, of which the best known arises from \hbox{$(A_1,A_2,A_3)=(1,2,3)$} with $H_o=6$
and then as above $n=(6t+7)(12t+13)(18t+19)$ (Chernick equivalently has $(6M+1)(12M+1)(18M+1))$.

For $d=3$, let $A:=A_1$, $B:=A_2$, $C:=A_3$.
Then, for any $K_3N$, from (\ref{B3}) $A$, $B$, $C$ are pairwise coprime, $\ell=ABC$, and so from (\ref{B4})
there exists a positive integer $F$ such that
\begin{subequations}\label{B5}
\begin{equation}H(AB+AC+BC)+A+B+C=FABC,\qquad\text{i.e.}\label{B5a}
\end{equation}
\begin{equation}
F=H\Bigl(\frac{1}{A}+\frac{1}{B}+\frac{1}{C}\Bigr)+\frac{1}{AB}+\frac{1}{AC}+\frac{1}{BC}
\quad\text{is a positive integer.}\label{B5b}\end{equation}
\end{subequations}

Also given any pairwise coprime $A$, $B$, $C$, then $\sum AB$ and $ABC$ are coprime, so 
(2.5a) has a unique solution
for $F$ and $H_o$, and we get a $K_3$-family as described above.
Putting $p_i{}'=A_iH$ in 
(2.2) with $d=3$, with 
(2.5a) we have \begin{equation}\label{B6}n'=ABCH(H^2+F)\end{equation}
\begin{theorem}\label{Th2_1}
For any $K_3N$, $1\leq F\leq2H$.
\end{theorem}
\begin{proof}
Based on 
(2.5b), write $F=F(A,B,C,H)$. We have $H\geq 2$ and $B\geq 2$. Obviously $F\geq 1$.
We consider two cases: (a) $B\geq 3$, (b) $B=2$.

(a) For $B\geq 3, F(A,B,C,H)\leq F(1,3,4,H)=\dfrac{19H}{12}+\dfrac{2}{3}< 2H\text{  for  }H\geq 2\text{   ,}$

(b) B=2. For $C\geq 7 \text{ we have } F(1,2,C,H)\leq F(1,2,7,H)=\dfrac{23H}{14}+\dfrac{5}{7}\leq 2H \text{ for } H\geq 2$.
Also \begin{equation*}F(1,2,5,H)<F(1,2,3,H)=\dfrac{11H}{6}+1\leq 2H\text{  for  }H\geq 6\\.\end{equation*}
But for $(A,B,C)=(1,2,3) \text{ or } (1,2,5) \text{ we have } H_0=6$, whence the result, with $F=2H$ only for $n=7 \cdot 13 \cdot 19=1729$.
\end{proof}

Note: $(A,B,C,H)=(1,2,7,2) \text{ yields } n=3 \cdot 5 \cdot 15 =225$, not a $K_3N$ since $E=1$.


\subsection{The equation(s) of Beeger, Duparc and Pinch.}\label{Sec2c}
For any $KN$ we have $n=Pqr=qQ=rR$, and hence $Q=Pr$ and $R=Pq$.\newline
Then from (\ref{B1}) $Dq\,'=Q\,'=(Pr)'=Pr'+P\,'$,
so \begin{equation}\label{B7}Dq\/'-Pr'=P\,'\text{ , and similarly  } Er'-Pq\/'=P\,'\\.\end{equation}
Solving simultaneously for $q\/'$ and $r'$, and writing
\begin{equation}\label{B8}\Delta:=DE-P\,^2,\qquad\text{we get}\end{equation}
\vskip-12pt
\begin{subequations}\label{B9}
\begin{equation}\label{B9a}q\/'=\frac{P\,'(P+E)}{\Delta}\qquad\text{and}
\end{equation}
\begin{equation}\label{B9b}r'=\frac{P\,'(P+D)}{\Delta}\\.
\end{equation}
\end{subequations}
Beeger for $d=3$ in~1950~\cite{Beeger7} and Duparc for $d\geq3$ in~1952~\cite{Duparc8} gave~(\ref{B9a}),
and Pinch~\cite{Pinch} bases his first algorithm on~(\ref{B9}).
From~(\ref{B9}), $\Delta\geq1$.
Also Duparc showed
\begin{theorem}\label{Th2_2}For any $KN$, $2\leq E\leq P-1$.\end{theorem}
\begin{proof}
From the definition of a $KN$, $E\geq 2$. Also $r-q-1\geq 1$. Hence from~(\ref{B7})
\begin{equation*}E=\frac{Pq\/'+P\,'}{r'}=\frac{Pr'-P(r'-q\/')+P\,'}{r'}=P-\frac{P(r-q-1)+1}{r'}<P\text{ ,}\end{equation*}
whence the result.\end{proof}

The following equations based on~(\ref{B8}) and~(\ref{B9}) will also be useful in~\S3 and \S4.
Define $s:=P-E$ and $\eta:=D-P-s$, so
\begin{equation}\label{B10}E=P-s\text{  and  }D=P+s+\eta\end{equation}
Then from Theorem \ref{Th2_2}, $1\leq s\leq P-2$, and from~(\ref{B8}) we have
\begin{equation}\Delta=\eta(P-s)-s^2\text{ ,  whence    }\eta=\frac{\Delta+s^2}{P-s}=\frac{\Delta+s^2}{E}
\text{  and so  }\eta\geq 1\label{B11}\end{equation}
Hence
\begin{equation}\label{B12}s^2+\eta s=\eta P-\Delta\text{,}\quad\text{so}\quad\Delta<\eta P\text{,}\quad\text{and}\quad
s=\sqrt{\eta P - \Delta+\frac{\eta^2}{4}}-\frac{\eta}{2}\end{equation}
Hence if
\begin{subequations}\label{B13}
\begin{equation}\label{B13a}\theta:=\sqrt{4(\eta P-\Delta)+\eta^2}\text{ ,}\qquad\text{then}\end{equation}
\begin{equation}\label{B13b}s=\frac{\theta-\eta}{2}\text{ ,}\quad E=P-\frac{\theta-\eta}{2}
\text{ ,}\quad D=P+\frac{\theta+\eta}{2}\end{equation}
\end{subequations}
Then~(\ref{B9}) becomes
\begin{subequations}\label{B14}
\begin{equation}\label{B14a}q=\frac{P\,'}{\Delta}\Bigl(2P+\frac{\eta-\theta}{2}\Bigr)+1\text{ ,}\end{equation}
\begin{equation}\label{B14b}r=\frac{P\,'}{\Delta}\Bigl(2P+\frac{\eta+\theta}{2}\Bigr)+1\text{ ,}\end{equation}
\end{subequations}
\begin{equation*}\shoveleft{\text{whence}\quad qr=\frac{P'{\,}^2}{\Delta^2}\biggl\{\Bigl(2P+\frac{\eta}{2}\Bigr)^2
-\frac{\theta{\,}^2}{4}\biggr\}+\frac{P\,'}{\Delta}(4P+\eta)+1}\end{equation*}
and using~(\ref{B13a}) for $\theta^2$ we get
\begin{equation}\label{B15}n=Pqr=P\Bigl\{\frac{P'{\,}^2}{\Delta^2}(4P\,^2+\eta P+\Delta)+\frac{P\,'}{\Delta}(4P+\eta)+1\Bigr\}\end{equation}

Subject to the solubility of certain congruences, a particular choice of $\Delta$ and $\eta$ leads to one or more $K_3$-families
via~(\ref{B12},~\ref{B13},~\ref{B14}), using~(\ref{B12}) with $s$ as a first parameter.


\subsection{Relations connecting $K_3N$ variables}\label{Sec2d}
These will be needed in \S4.

\begin{equation*}\text{We have}\quad E=\frac{R\,'}{r'}=\frac{(pq)'}{r'}
=\frac{p\,'q\,'+p\,'+q\,'}{r'}=\frac{ABH+A+B}{C}\text{ ,}\end{equation*}
since $p_i\,'=A_iH$, and hence
\begin{equation}\label{B16}C=\frac{ABH+A+B}{E}=\frac{Bp+A}{E}\text{ .}\end{equation}
Combining this with (\ref{B5a}) we get $ABCF=CE+CH(A+B)+C,\quad$i.e.
\begin{equation}\label{B17}\quad ABF=(A+B)H+E+1,\qquad\text{whence, writing}\end{equation}
\begin{equation}\label{B18}G:=AF-H,\qquad\text{we have}\end{equation}
\begin{equation}\label{B19}B=\frac{AH+E+1}{G}=\frac{p+E}{G}\text{ .}\end{equation}

\begin{equation*}\text{Also using (\ref{B9a}), }BH=q\,'=\frac{p\,'(p+E)}{\Delta}=\frac{AH(p+E)}{\Delta}
\text{ ,  so  }B=\frac{A(p+E)}{\Delta}\text{ ,}\end{equation*}
whence from (\ref{B19}),\begin{equation}\label{B20}\Delta=AG\text{ .}\end{equation}
So $G\geq 1$, we have $\Delta\geq 1$, and we now show
\begin{theorem}\label{Th2_3}For any $K_3N$, $1\leq G<2H$ and $1\leq\Delta<2p-2$.\end{theorem}
\begin{proof}
We have $p<q\,'$ and, from Theorem \ref{Th2_2}, $E\leq p\,'$.
Hence, using (\ref{B20}) and~(\ref{B9a}),
\begin{equation*}AG=\Delta=\frac{p\,'(p+E)}{q\,'}<\frac{p\,'(2p-1)}{p}
=2p-3+\frac{1}{p}=2AH-1+\frac{1}{p}<2AH=2p\,'\text{ ,}\end{equation*}
from which both statements in Theorem \ref{Th2_3} follow.
\end{proof}

We observe that (\ref{B18},~\ref{B19},~\ref{B16}) enable us to express in turn $A$, $p$, $B$, $q$, $C$, $r$ and $n$ 
in terms of $E$, $F$, $G$ and $H$, which we shall exploit later.


\subsection{Bounds, variables and challenges}\label{Sec2e}
The next two sections are concerned with finding inequalities $y\leq f(x)$ showing upper bounds for $y$ given $x$,
where $x$ and $y$ are variables connected with $KN$'s and hence $CN$'s.
Here $f$ is an increasing function and, usually,
$x$ is $P$, $p$ or $n$: for example, we shall show that an upper bound
for $r$ given $n$ is given by $r\leq \lceil\sqrt{\frac{n}{2}}\rceil$.
Obviously, if $f^{-1}$ exists, any such relation is equivalent to a lower bound of $f^{-1}(y)$ for $x$ given $y$. 
Also, if $x$ is $n$ and we are looking for all $CN$'s less than some large $X$, we have $y\leq f(X)$.
Our symbols have been defined as integer variables connected with $KN$'s,
but sometimes we shall treat them as real variables obeying the same inequalities as the integer variables.
If possible we shall exhibit $K$-families for which the bound is attained,
or failing that some $CN$'s or $KN$'s for which it is approached.
Some of these bounds were used in executing our algorithm for $C_3N$'s (see \S5),
although invariably a weaker (and more easily proved) bound would have sufficed.

Challenges: occasionally I offer a challenge to find a $C_3N$ or $K_3N$ satisfying certain conditions,
and I would be very interested to receive successful responses at my address at the end of this paper.


\section{Bounds for $KN$ variables with $d\geq 3$}\label{Sec3}

\subsection{Upper bounds given $P$ for $q$, $r$ and $n$}\label{Sec3a}
\begin{theorem}\label{Th3_1}(Duparc's theorem) For any $KN$,
\begin{equation*}q\leq\Bigl(P-1\Bigr)\Bigl(2P+\frac{1}{2}-\sqrt{P-\frac{3}{4}}\Bigr)+1\text{ .}\end{equation*}\end{theorem}
\begin{proof}
From (\ref{B9a}) and Theorem \ref{Th2_2}, if $\Delta\geq2$ we have
\begin{equation}\label{C1}q\/'=\frac{P\,'(P+E)}{\Delta}\leq\frac{P\,'(P+P\,')}{2}=P\,'\Bigl(P-\frac{1}{2}\Bigr)\end{equation}
Fix $\Delta$ and $P$, with $\Delta<P$, and regard $\theta = \theta(\eta)$ and $q=q(\eta)$ as functions of a real variable
 $\eta$, defined by (\ref{B13a}) and (\ref{B14a}) respectively, with $\eta = \eta_o$ for the KN under consideration.
Then we have 
\begin{equation}\label{C2}
\frac{d}{d\eta}(\eta-\theta)=1-\frac{\eta+2P}{\sqrt{(\eta+2P)^2-4(P\,^2+\Delta)}}<0\text{ ,}
\end{equation}
so as $\eta$ increases, $\eta-\theta$ and hence $q$ both decrease, and hence $q(\eta_o) \leq q(1)$.
We shall use (\ref{C2}) for general $\Delta$ later, but with $\Delta=1$, $q(1)=P\,'(2P+\frac{1}{2}-\sqrt{P-\frac{3}{4}})+1$.
But for $P\geq 3$, $P\,'(P-\frac{1}{2})<P\,'(2P+\frac{1}{2}-\sqrt{P-\frac{3}{4}})$,
so with (\ref{C1}) the result follows.\end{proof}

We note that for this maximal $q$ situation, with $\Delta=\eta=1$, (\ref{B15})~gives
\begin{equation}\label{C3}n=N_2(P):=4P\,^5-7P\,^4+7P\,^3-4P\,^2+P\end{equation}

Also from (\ref{B12}, \ref{B13}, \ref{B14}) we get
\begin{equation}\label{C4}\begin{split}P=&p_2(s):=s^2+s+1\text{ ,}\quad\theta=2s+1\\
q=Q_2(P):=\Bigl(P-1\Bigr)\Bigl(2P+\frac{1}{2}-\sqrt{P-\frac{3}{4}}\Bigr)+1=&q_2(s):=2s^4+3s^3+3s^2+2s+1\\
r=R_2(P):=\Bigl(P-1\Bigr)\Bigl(2P+\frac{1}{2}+\sqrt{P-\frac{3}{4}}\Bigr)+1=&r_2(s):=2s^4+5s^3+6s^2+3s+1\end{split}
\end{equation}
and we have the $q$-maximal $K_3$-family $n=n_2(s):=p_2(s) \cdot q_2(s) \cdot r_2(s)$.
A computer search by Ian Williams (see \S\ref{Sec7}) for $1\leq s\leq 4906$
found just 12 $q$-maximal $C_3N$'s, with $s=1,2,6,90,\ldots$ up to 3654, and only one
 $q$-maximal $C_dN$ with $d>3$, namely the $C_6N$ 
$n_2(1493)=7\cdot 19\cdot 31\cdot 541\cdot 9947309489407\cdot 9953972118361\bumpeq 2.2\times 10^{32}$, 
with $P=2230543$.

\begin {theorem}\label{Th3_2}For any $KN$, $r\leq R_3(P):=\frac{1}{2}(P-1)(P+1)^2+1$, with equality if{}f $q=Q_3(P):=P\,^2+P-1$.
\end{theorem}
\begin{proof}
From (\ref{B8}) and (\ref{B9b}),
\begin{equation*}r'=\frac{P\,'(P+D)}{\Delta}=\frac{P\,'}{\Delta}\Bigl(P+\frac{P\,^2+\Delta}{E}\Bigr)
=P\,'\Bigl(\frac{P}{\Delta}+\frac{P\,^2}{\Delta E}+\frac{1}{E}\Bigr)\text{ .}\end{equation*}
But $\Delta\geq 1$ and $E\geq 2$, so
\begin{equation*}r'\leq P\,'\Bigl(P+\frac{P\,^2}{2}+\frac{1}{2}\Bigr)=\frac{1}{2}P\,'(P+1)^2\text{ ,}\end{equation*}
with equality if{}f $\Delta=1$ and $E=2$,
\, which from~(\ref{B9a}) and~(\ref{C1}) occurs if{}f \mbox{$q\/'=P\,'(P+2)$,} whence the result.\end{proof}

\begin{theorem}\label{Th3_3}For any $KN$,
$n\leq N_3(P):=\frac{1}{2}(P\,^6+2P\,^5-P\,^4-P\,^3+2P\,^2-P)$,
with equality if{}f
$q=Q_3(P)=P\,^2+P-1$ and $r=R_3(P)=\frac{1}{2}(P-1)(P+1)^2+1$.\end{theorem}
\begin{proof}
For $\Delta\geq 2$, from (\ref{C1}), $q\leq P\,^2-\frac{3}{2}P+\frac{3}{2}<Q_3(P)$ for $P\geq 3$,
and from Theorem \ref{Th3_2}, $r<R_3(P)$, so $n<P \cdot Q_3(P) \cdot R_3(P)$.

For $\Delta=1$ and given $P$, with $\eta$ and $n$ real variables, from (\ref{B15}) $n$ is greatest when $\eta$ is greatest.
But from (\ref{B11})
\begin{equation*}\eta=\frac{\Delta+s^2}{P-s}=\frac{s^2+1}{P-s}\text{ ,}\end{equation*}
and $1\leq s\leq P-2$, giving maximum $\eta=\frac{1}{2}(P\,^2-4P+5)$ when $s=P-2$, i.e. $E=2$ from (\ref{B10}).
But from Theorem \ref{Th3_2}, for $\Delta=1$ and $E=2$ we have $q=Q_3(P)$ and $r=R_3(P)$, so $n\leq P \cdot Q_3(P) \cdot R_3(P)$,
which multiplies out to give the result.\end{proof}

Beeger \cite{Beeger7} for $d=3$ and Duparc \cite{Duparc8} for any $CN$ proved results similar to Theorem \ref{Th3_1}
and weaker than Theorem \ref{Th3_2}.
They were chiefly concerned to show that the number of $CN$'s for given $P$ is finite.
Swift stated the first result of Theorem \ref{Th3_3} for $d=3$ in 1975, but his proof is not published in \cite{Swift9}.

For a $K_3$-family attaining these upper bounds for $r$ and~$n$ given~$P$,
we simply put $P=2t+1$ in $n=N_3(P)=P \cdot Q_3(P) \cdot R_3(P)$.
We note that $Q_2(3)=Q_3(3)=11$ and $R_2(3)=R_3(3)=17$, so the smallest $CN$ (and $KN$, easily shown),
\hbox{$n=561$,} uniquely is $q$-maximal, $r$-maximal and $n$-maximal.
$N_3(P)$ gives $C_3N$'s 
for \hbox{$P=3, 5, 31, 41, 83, 131,\ldots $;} and $N_3(65)=5 \cdot 13 \cdot 4289 \cdot 139393=38860677505$
is the smallest $r$-$n$-maximal $C_4N$ for given $P$. 
A computer search over $3\leq P\leq 132425$ found 178 $C_3N$'s, 18 $C_4N$'s, 29 $C_5N$'s and 9 $C_6N$'s 
which are $r$-$n$-maximal, the largest of which is the $C_3N$ with $P=131711$.

\newpage

\subsection{Upper bounds given $n$ for $P$, $q$, $r$}\label{Sec3b}
For an upper bound for $P$ given $n$, we do not attempt to improve on the obvious:

\begin{theorem}\label{Th3_4}For any $K_dN$,
$P<n^{\frac{d-2}{d}}$\end{theorem}
\begin{proof}
We have $p_1<p\,_2<\ldots<p\,_{d-2}<q<r$ (with obvious modification for $d<5$).
Hence \(P=\displaystyle\prod_{i=1}^{d-2}p_i\leq p_{d-2}^{\;{d-2}},\text{ so }p\,_{d-2}\geq P^{\frac{1}{d-2}}\text{  and  }
P=\dfrac{n}{qr}<\dfrac{n}{p_{d-2}^{\;2}}\leq \dfrac{n}{P^{\frac{2}{d-2}}}.\)
Hence $P^{\frac{d}{d-2}}<n$ and the result follows.\end{proof}

It seems plausible that an upper bound for $q$ given $n$ should correspond to the Theorem \ref{Th3_1} bound for $q$ given $P$,
so in terms of (\ref{C3}) and (\ref{C4}) we should have
\begin{theorem}\label{Th3_5}For any $KN$,
$ q\leq Q_2(N_2^{-1}(n))=q_2(n_2^{-1}(n))$, with equality if{}f 
\begin{equation*}q=Q_2(P)=\Bigl(P-1\Bigr)\Bigl(2P+\frac{1}{2}-\sqrt{P-\frac{3}{4}}\Bigr)+1\end{equation*}

Explicitly,
\begin{equation*}q\leq\biggl\lfloor\sqrt[5]{2n^2}-\sqrt[10]{\frac{n^3}{64}}-\frac{1}{10}\sqrt[5]{\frac{n}{4}}
+\frac{17}{20}\sqrt[10]{\frac{n}{4}}\biggr\rfloor\end{equation*}
\end{theorem}

\begin{proof}
If $V=N_2^{-1}(n)$, the substitution $V=v^2+v+1$ with the algebra of (\ref{C3}) and (\ref{C4}) establishes the equivalence
 of the functions $Q_2N_2^{-1}$ and $q_2n_2^{-1}$, and if $\Delta =\eta =1$ it is immediate that 
$q=Q_2(N_2^{-1}(n))$; also $\Delta =\eta =1$ for $n=561$, the only $KN$ with $P=3$.
So we assume $KN$'s with $\Delta \eta \geq 2$ and $P \geq 5$, and we consider two cases: (i) $\Delta \geq P\,'$,
(ii) $1 \leq \Delta \leq P-2$. We write $\lambda := \frac{P\,'}{\Delta}$.

(i) $\lambda \leq 1$. From (\ref{B13b}) and (\ref{B14a}), $q\/'=\lambda (2P-s)<2\lambda P$, and then 
from (\ref{B15}) 
$n=P\{\lambda^2(4P^2+\eta P+\Delta)+\lambda(4P+\eta)+1\}>4\lambda^2P^3>4\lambda^2(\frac{q\/'}{2\lambda})^3=\frac{q\/'^3}{2\lambda}$.
Hence $q\/'^3<2\lambda n\leq 2n$, so $q<(2n)^{\frac{1}{3}}+1$. Let $v:=n_2^{-1}(n)$, so $n=n_2(v)$, 
and it is easily verified that $(2n_2(v))^{\frac{1}{3}}+1<q_2(v)$ for $v\geq 1.1$. But $n_2(1.1)<963$ 
and every $KN$ with $P\geq 5$ has $n\geq 1105=5\cdot 13 \cdot 17$ and hence $v\geq 1.1$, so $q<q_2(n_2^{-1}(n))$ 
as required.

(ii) $\lambda >1$. Let $n_o$ be any particular $KN$ with associated $K$-variable values 
$P_o, q_o, n_o, \Delta_o, \eta_o, \theta_o, s_o$ and $\lambda_o=\frac{P_o\,'}{\Delta_o}>1$.
We define $s, \theta, q, r$ and $n$ as functions of independent real variables $P, \Delta, \eta$ by the formulae 
of (\ref{B12}-\ref{B15}) over the domain $3\leq P\leq P_o$, $1\leq \eta \leq \eta_o$ and $1\leq \Delta \leq \Delta_o$ 
with $\Delta < P$ (ensuring $\theta \in \Re$), and we write $n=n(P,\Delta,\eta)$ with $n_o=n(P_o,\Delta_o,\eta_o)$, etc. 
Then, if $\eta_o>1$, over the interval $1<\eta <\eta_o$ from (\ref{B14a}), (\ref{C2}) and (\ref{B15}) we have 
$\frac{\partial q}{\partial \eta}<0$ and $\frac{\partial n}{\partial \eta}>0$,
whence $q_a:=q(P_o,\Delta_o,1)\geq q_o$ and $n_a:=n(P_o,\Delta_o,1)\leq n_o$, with equality if{}f $\eta_o=1$ .
In what follows we use the suffices ``a" and ``b" to correspond to``states" $(\Delta,\eta)=(\Delta_o,1)$ and 
$(\Delta,\eta)=(1,1)$ respectively.

Then, if $\Delta_o>1$, for $1\leq\Delta\leq\Delta_o$ we define the function $P=P^*(\Delta)$ implicitly via the relation 
$n(P,\Delta,1)=n_a$, so $n_a=n_b$ and
\begin{equation}
 \label{C5}P\Bigl{\{}\frac{P\,'\,^2}{\Delta^2}(4P^2+P)+5\frac{P\,'P}{\Delta}+1\Bigr{\}}=P\{\lambda^2P(4P+1)+5\lambda P+1\}=n_a,\text{ constant,}
\end{equation}
and as $\Delta$ decreases from $\Delta_o$ to 1 we see from the first of these expressions that P steadily decreases, 
whence from the second $\lambda$ steadily increases. So, for $\Delta_o>\Delta >1$, 
$\frac{P\,'}{\Delta}=\lambda >\lambda_o >1$, whence $P-\Delta >1$, so $\theta = \sqrt{4(P-\Delta)+1}>\sqrt{5}$, 
ensuring that, via (\ref{B13}) and (\ref{B14}) with $\eta =1$, $\theta^*(\Delta):=\theta(P^*(\Delta),\Delta,1)=\theta$, 
and, similarly, $q^*(\Delta)=q$ and $r^*(\Delta)=r$ are defined for $1\leq \Delta\leq\Delta_o$; also 
$P>\Delta+1>2$, so certainly $N_2(P)$ is an increasing function over the relevant domain. Now, for $\Delta_o\geq1$, 
write $P_b:=P^*(1)=N_2^{-1}(n_a)$, so since by hypothesis $P_o\geq 5$ and $\lambda_o>1$, we have 
\begin{align*}
N_2(P_b)=&n(P_b,1,1)=n_a=n(P_o,\Delta_o,1)
\\=&P_o\{\lambda_o^2(4P_o^2+P_o)+5\lambda_oP_o+1\}
>4P_o^3+6P_o^2+P_o\geq 655>561=N_2(3),
\end{align*}
whence $P_b>3$; but $P\geq P^*(1)=P_b$, so $P>3$.
Thus from the above we have
\begin{equation}
 \label{C6}\text{For }1<\Delta<\Delta_o\quad,\quad\lambda>1\quad,\quad P>3 \text{ and }\theta>\sqrt{5}.
\end{equation}
Also let $s_b:=p_2^{-1}(P_b)=\frac{1}{2}(\sqrt{4P_b-3}-1)$ and $q_b:=Q_2(P_b)=q_2(s_b)$, so if $\Delta_o=1$ 
then $q_b=q_a$.

Next we shall show that, with $\eta=1$, $\frac{dq}{d\Delta}<0$ for $1<\Delta<\Delta_o$, and then since $Q_2,q_2,N_2$ and $n_2$ 
are all increasing functions over the relevant domains, we shall have $q^*(\Delta_o)\leq q^*(1)$, and so
$q_o\leq q_a=q(P_o,\Delta_o,1)=q^*(\Delta_o)\leq q^*(1)=q(P_b,1,1)=q_b=Q_2(P_b)=Q_2(N_2^{-1}(n_a))\leq Q_2(N_2^{-1}(n_o))$, 
and likewise $q_o\leq q_b=q_2(s_b)\leq q_2(n_2^{-1}(n_o))$, with equality throughout if{}f $\Delta_o=\eta_o=1$, 
i.e. if{}f $q_o=Q_2(P_o)$ as required.

From the first of the two expressions for $n_a$ in (\ref{C5}) we have
\begin{equation}\label{C7}qr=\frac{P'\,^2(4P\,^2+P)}{\Delta^2}+\frac{5(P\,^2-P)}{\Delta}+1\end{equation}
\begin{equation*}\begin{split}\therefore \frac{d(qr)}{d\Delta}=\frac{1}{\Delta^2}\{2P\,'(4P\,^2+P)+P'\,^2(8P+1)\}\frac{dP}{d\Delta}
-\frac{2P'\,^2(4P\,^2+P)}{\Delta^3}\\+\frac{5(2P-1)}{\Delta}\frac{dP}{d\Delta}-\frac{5PP\,'}{\Delta^2}\text{ ,}\quad\text{that is}
\end{split}\end{equation*}
\begin{equation}\label{C8}\quad\frac{d(qr)}{d\Delta}=\Bigl\{\frac{P\,'(16P\,^2-5P-1)}{\Delta^2}
+\frac{5(2P-1)}{\Delta}\Bigr\}\frac{dP}{d\Delta}
-\frac{2PP'\,^2(4P+1)}{\Delta^3}-\frac{5PP\,'}{\Delta^2}\end{equation}
\begin{equation*}\text{But $Pqr=n_a$, constant, so}\quad
qr\frac{dP}{d\Delta}+P\frac{d(qr)}{d\Delta}=0\text{ ;}\qquad\qquad\end{equation*}
substituting from (\ref{C7}) and (\ref{C8}) gives
\begin{align*}\Bigl\{\frac{P'\,^2P(4P+1)}{\Delta^2}+\frac{5PP\,'}{\Delta}+1\text{ }+\text{ }&\frac{PP\,'(16P\,^2-5P-1)}{\Delta^2}
+\frac{5P(2P-1)}{\Delta}\Bigr\}\frac{dP}{d\Delta}\\=&\frac{2P{\,}^2P'\,^2(4P+1)}{\Delta^3}+\frac{5P\,^2P\,'}{\Delta^2}\text{ ,}
\quad\text{whence}\end{align*}
\begin{equation}\label{C9}\Delta\Bigl\{2P\,'(10P\,^2-4P-1)+5\Delta(3P-2)+\frac{\Delta^2}{P}\Bigr\}\frac{dP}{d\Delta}
=2PP'\,^2(4P+1)+5\Delta PP\,'\end{equation}

Also with $\eta=1$, (\ref{B14a}) gives $2q=\frac{P\,'}{\Delta}(4P+1-\theta)+2$ and from (\ref{B13a}) $\theta^2=4P+1-4\Delta$,
so $\theta=\frac{1}{\theta}(4P+1-4\Delta)$,
\begin{equation*}\frac{d\theta}{d\Delta}=\frac{2}{\theta}\Bigl(\frac{dP}{d\Delta}-1\Bigr)\quad\text{and}\quad
2\frac{dq}{d\Delta}=\Bigl(\frac{1}{\Delta}\frac{dP}{d\Delta}-\frac{P\,'}{\Delta^2}\Bigr)\Bigl(4P+1-\theta\Bigr)
+\frac{P\,'}{\Delta}\Bigl(4\frac{dP}{d\Delta}-\frac{d\theta}{d\Delta}\Bigr)\text{ .}\end{equation*}
\begin{equation*}\text{Hence }
2\Delta^2\frac{dq}{d\Delta}
=\Bigl(\Delta\frac{dP}{d\Delta}-P\,'\Bigr)\Bigl(4P+1-\frac{4P+1}{\theta}+\frac{4\Delta}{\theta}\Bigr)
+\Delta P\,'\Bigl(4\frac{dP}{d\Delta}-\frac{2}{\theta}\frac{dP}{d\Delta}+\frac{2}{\theta}\Bigr),
\text{ i.e.}\end{equation*}
\begin{equation}\label{C10}2\Delta^2\frac{dq}{d\Delta}=
\Delta\Bigl\{8P-3-\frac{6P-1}{\theta}+\frac{4\Delta}{\theta}\Bigr\}\frac{dP}{d\Delta}
-P\,'\Bigl(4P+1\Bigr)\Bigl(1-\frac{1}{\theta}\Bigr)-\frac{2\Delta P\,'}{\theta}\end{equation}

Writing $M:=2P\,'(10P\,^2-4P-1)+5\Delta(3P-2)+\frac{\Delta^2}{P}$, from (\ref{C9}) and (\ref{C10})\newline we get
\begin{equation*}\begin{split}-2\Delta^2M\frac{dq}{d\Delta}=\Bigl\{P\,'\Bigl(4P+1\Bigr)\Bigl(1-\frac{1}{\theta}\Bigr)+\frac{2\Delta P\,'}{\theta}\Bigr\}
\Bigl\{2P\,'(10P\,^2-4P-1)+5\Delta(3P-2)+\frac{\Delta^2}{P}\Bigr\}\\
-\Bigl\{8P-3-\frac{6P-1}{\theta}+\frac{4\Delta}{\theta}\Bigr\}\Bigl\{2PP'\,^2(4P+1)+5\Delta PP\,'\Bigr\}
\end{split}\end{equation*}
The coefficient of $\dfrac{\Delta}{\theta}$ in this expression is 
\begin{equation*}\begin{split}
 &P'\{4P'(10P^2-4P-1)-5(4P+1)(3P-2)+5P(6P-1)-8PP'(4P+1)\}\\
&=P'\{4P'(2P^2-6P-1)-5(6P^2-4P-2)\}=2P'\,^2\{2(2P^2-6p-1)-5(3P+1)\}\\
&=2P'\,^2(4P^2-27P-7)=2P'\,^2(4P+1)(P-7)
\end{split}\end{equation*}
and of $\dfrac{\Delta^2}{\theta}$ is
\begin{equation*}
 P'\Bigl\{-\frac{4P+1}{P}+10(3P-2)-20P\Bigr\}=\frac{P'(10P^2-24P-1)}{P},
\end{equation*}
so
\begin{multline*}
 -2\Delta^2M\frac{dq}{d\Delta}=2P'\,^2(4P+1)(2P^2-P-1)+5\Delta P'(4P^2-2P-2)+\frac{\Delta^2P'(4P+1)}{P}\\
-\frac{2P'\,^2(4P+1)}{\theta}\{4P^2-3P-1\}+\frac{2\Delta P'\,^2(4P+1)(P-7)}{\theta}\\
+\frac{\Delta^2 P'(10P^2-24P-1)}{\theta P}+\frac{2\Delta^3 P'}{\theta P}
\end{multline*}

Factorising terms and dividing by $2P'\,^3(4P+1)$ gives
\begin{equation*}\begin{split}-\frac{\Delta^2M}{(4P+1)P'\,^3}\frac{dq}{d\Delta}=W:=2P+1+\frac{5\Delta(2P+1)}{P\,'(4P+1)}
+\frac{\Delta^2}{2PP'\,^2}-\frac{(4P+1)}{\theta}\\
+\frac{\Delta(P-7)}{\theta P\,'}
+\frac{\Delta^2(10P\,^2-24P-1)}{2\theta PP'\,^2(4P+1)}
+\frac{\Delta^3}{\theta PP'\,^2(4P+1)}\end{split}
\end{equation*}

From (\ref{C6}) $P>3$ and $\theta>\sqrt 5$, so $M>0$ and also $10P^2-24P-1>0$, whence 
\begin{align*}
 W>&2P+1-\frac{4P+1}{\theta}+\frac{5(2P+1)}{\lambda(4P+1)}+\frac{P-7}{\lambda\theta}\\
=&\frac{(2\theta-4)P+\theta-1}{\theta}+\frac{(4P+1)(P-7)+5\theta(2P+1)}{\lambda\theta(4P+1)}\\
>&\frac{(4P+1)(P-7)+5\sqrt{5}(2P+1)}{\lambda\theta(4P+1)}=\frac{4P^2+(10\sqrt{5}-27)P+5\sqrt{5}-7}{\lambda\theta(4P+1)}>0
\end{align*}
since the discriminant of the numerator is negative.

Hence $\dfrac{dq}{d\Delta}=-\dfrac{(4P+1)\lambda^2P\,'W}{M}<0$ as required, completing the proof of the first part of the theorem.

We now outline a method of expressing $Q_2(N_2\,^{-1}(n))$ as a power series in $\frac{1}{\sqrt[10]n}$.
Let $U:=\sqrt[5]{\frac{n}{4}}$, and (\ref{C2}) becomes
\begin{equation}\label{C11}U^5=P\,^5-\frac{7}{4}P\,^4+\frac{7}{4}P\,^3-P\,^2+\frac{1}{4}P\end{equation}
and if we now put $P=U+B_o+\frac{B_1}{U}+\frac{B_2}{U^2}+\ldots$,
substitute into (\ref{C11}) and equate coefficients of $U^{1-k}$ to find successively $B_o$, $B_1$, $B_2$, $\ldots$,
we obtain
\begin{equation}\label{C12}P=U+\frac{7}{20}-\frac{21}{200U}+\frac{1}{250U^2}+\frac{2787}{160000U^3}+\ldots\end{equation}

To find
\begin{multline*}Q_2(N_2\,^{-1}(n))=Q_2(P)=P\,'\Bigl(2P+\frac{1}{2}-\sqrt{P-\frac{3}{4}}\Bigr)+1\\
=2PP\,'+\frac{P\,'}{2}-P\,'\sqrt{P-\frac{3}{4}}+1\text{ ,}\end{multline*}
and with $u=\sqrt U=\sqrt[10]{\frac{n}{4}}$, from (\ref{C12}) we have 
\begin{equation*}P=u^2+\frac{7}{20}-\frac{21}{200u^2}+\frac{1}{250u^4}+\ldots\text{,}\qquad
P\,'=u^2-\frac{13}{20}-\frac{21}{200u^2}+\frac{1}{250u^4}+\ldots\quad\text{and}\end{equation*}
\begin{multline*}\sqrt{P-\frac{3}{4}}=\Bigl(u^2-\frac{2}{5}-\frac{21}{200u^2}+\frac{1}{250u^4}+\ldots\Bigr)^{\frac{1}{2}}
=u\Bigl[1-\Bigl(\frac{2}{5u^2}+\frac{21}{200u^4}-\frac{1}{250u^6}+\ldots\Bigr)\Bigr]^{\frac{1}{2}}\\
=u-\frac{1}{5u}-\frac{29}{400u^3}-\frac{1}{80u^5}+\ldots\text{ ,}\end{multline*}
after applying the binomial series and simplifying.

Substituting back, we get
\begin{equation*}Q_2(P)=2u^4-u^3-\frac{u^2}{10}+\frac{17u}{20}-\frac{1}{5}+\frac{19}{400u}+\frac{53}{2000u^2}
-\frac{477}{8000u^3}+\ldots\end{equation*}

If now the variables all belong to a $KN$, then $q=Q_2(P)$ is an integer and\newline
\hbox{$q=\lfloor 2u^4-u^3-0.1u^2+0.85u\rfloor$,}
and Theorem~\ref{Th3_5} follows.
\end{proof}

Much more simply, we now establish an upper bound for $r$ given $n$; first we prove

\begin{theorem}\label{Th3_6}For any $KN$,
\begin{equation*}p_i=\frac{\lambda_i\,'+\sqrt{\lambda_i'\,^2+4\lambda_in}}{2\lambda_i}
=\Bigl\lceil\sqrt{\frac{n}{\lambda_i}}\Bigr\rceil
\end{equation*}
\end{theorem}
\begin{proof}
Using the notation of \S\ref{Sec2a}, we have
\begin{equation*}p_ip_i\,'=\frac{p_iP_i\,'}{\lambda_i}=\frac{n}{\lambda_i}-\frac{p_i}{\lambda_i}\text{,}
\quad\text{whence}\end{equation*}
\begin{equation}\label{C13}\frac{n}{\lambda_i}=p_i\,^2-\Bigl(1-\frac{1}{\lambda_i}\Bigr)p_i
\end{equation}
and $\lambda_i p_i\,^2-\lambda_i\,'p_i-n=0$, giving the first equality.\newline
Also $\lambda_i\geq E\geq 2$, so $0<1-\frac{1}{\lambda_i}<1$,
and from (\ref{C13}) we thus have
\begin{equation*}(p_i-1)^2<\frac{n}{\lambda_i}<p_i\,^2\text{,}\quad\text{so}\quad
p_i=\Bigl\lceil\sqrt{\frac{n}{\lambda_i}}\Bigr\rceil\end{equation*}
completing the theorem.\end{proof}

\begin{theorem}\label{Th3_7}For any $KN$,
\begin{equation*}r\leq\frac{\sqrt{8n+1}+1}{4}\quad\text{and}\quad r\leq\Bigl\lceil\sqrt{\frac{n}{2}}\Bigr\rceil,
\quad\text{with equality if{}f}\quad E=2\text{.}\end{equation*}\end{theorem}
\begin{proof}This follows easily from Theorem \ref{Th3_6}.

$n\geq 561$ and $n$ is odd, so with $i=d$ in Theorem \ref{Th3_6} and $\lambda_d=E\geq 3$, we have
\begin{equation*}r=\Bigl\lceil\sqrt{\frac{n}{E}}\Bigr\rceil<\sqrt{\frac{n}{E}}+1<\sqrt{\frac{n}{2}}<\Bigl\lceil\sqrt{\frac{n}{2}}\Bigr\rceil
\quad\text{and}\quad \sqrt{\frac{n}{2}}<\frac{\sqrt{8n+1}+1}{4}\text{ ;}
\end{equation*}
in conjunction with $E=2$ in Theorem \ref{Th3_6}, Theorem \ref{Th3_7} follows.
\end{proof}

The smallest $CN$ which is $r$-maximal for given $n$ but not for given $P$ 
(see Theorem \ref{Th3_2}) is $8911=7 \cdot 19 \cdot 67.$
Another of this type is $949803513811921=17 \cdot 31 \cdot 191 \cdot 433 \cdot 21792241,$
which Pinch in~\cite{Pinch} says contains the largest prime factor among $CN$'s $<10^{15}$.


\section{Bounds for $K_3N$ variables}\label{Sec4}
\subsection{Upper bounds given $p$ for $A$, $B$, $C$}\label{Sec4a}
To establish an upper bound for $A$ given $p$, we need one for $A$ given $H$:
\begin{theorem}\label{Th4_1}For any $K_3N$, $A<3H-\sqrt{\frac{H}{2}}$.
\end{theorem}
\begin{proof}
Suppose for some $K_3N$ that $A\geq\lambda H$ for some $\lambda>0$. Then (\ref{B5b}) yields
\begin{multline*}F\leq H\Bigl(\frac{1}{\lambda H}+\frac{1}{\lambda H+1}+\frac{1}{\lambda H+2}\Bigr)
+\frac{1}{\lambda H(\lambda H+1)}+\frac{1}{\lambda H(\lambda H+2)}+\frac{1}{(\lambda H+1)(\lambda H+2)}\\
=\frac{1}{\lambda}+\Bigl(\frac{1}{\lambda}-\frac{1}{\lambda(\lambda H+1)}\Bigr)
+\Bigl(\frac{1}{\lambda}-\frac{2}{\lambda(\lambda H+2)}\Bigr)
+\frac{3\lambda H+3}{\lambda H(\lambda H+1)(\lambda H+2)}\\
=\frac{3}{\lambda}-\frac{3\lambda H(H-1)+4H-3}{\lambda H(\lambda H +1)(\lambda H+2)}<\frac{3}{\lambda}\quad\text{since}\quad H\geq 2
\end{multline*}

But $F\geq 1$, so putting $\lambda=3$ we get $A<3H$.
Also if $\lambda=\frac{3}{2}$ then $F<2$, whence if $\frac{3H}{2}\leq A< 3H$, then $F=1$.
We define a big$A$-$K_3N$ to be a $K_3N$ with $A\geq \frac{3H}{2}$, and likewise a big$A$-$C_3N$.
All other $K_3N$'s obviously obey Theorem \ref{Th4_1}.


So we now put $A=3H-a$, with $1\leq a \leq \frac{3H}{2}$, and using $F=1$ we shall show that for given $a$,
$H<2a^2$, yielding Theorem \ref{Th4_1}. We write $\alpha:=-a$, \hbox{$A=3H+\alpha$,} 
\hbox{$B=3H+\beta$,} $C=3H+\gamma$, $\sigma:=\frac{\beta+\gamma}{2}$,
$\tau:=\frac{\gamma-\beta}{2}$, so $\beta\gamma=\sigma^2-\tau^2$, and \hbox{$S:=\sum \alpha=2\sigma-a$.}
Then for a $K_3N$ we have $-a=\alpha<\beta=\sigma-\tau<\gamma$, so~$0<\tau<a+\sigma$, and from~(\ref{B5})
\begin{equation*}F=\frac{(\sum AB)H+\sum A}{ABC}=\frac{27H^3+6(\sum\alpha)H^2+(\sum\alpha\beta+9)H+\sum\alpha}
{27H^3+9(\sum\alpha)H^2+3(\sum\alpha\beta)H+\alpha\beta\gamma}=1-\frac{m}{ABC}
\end{equation*}
where
\begin{subequations}\label{D1}
\begin{equation}\label{D1a}m:=m(H):=3(\sum\alpha)H^2+(2\sum\alpha\beta-9)H+\alpha\beta\gamma-\sum\alpha
\end{equation}
\begin{equation}\label{D1b}\quad =3SH^2+(2\sigma^2-2\tau^2-4a\sigma-9)H-a(\sigma^2-\tau^2-1)-2\sigma
\end{equation}
\end{subequations}
So for any $K_3N$, $F=1$ if{}f $m=0$, and $A=3H-a>0$, so $H>\frac{a}{3}$;
also for any big$A$-$K_3N$, $\gamma>0$ (since otherwise
\begin{equation*}\frac{H}{C}=\frac{H}{3H+\gamma}\geq\frac{1}{3}\quad\text{and so}\quad
F>\sum\frac{H}{3H+\alpha}>1\text{ , but}\quad F=1\text{).}\end{equation*}

\begin{equation*}\text{We now regard }m(H)\text{ and } F(H):=1-\frac{m(H)}{(3H+\alpha)(3H+\beta)(3H+\gamma)}\text{ as functions}
\end{equation*}
of an unrestricted real variable $H$, where $\alpha,\beta,\gamma$ are real, $\alpha<0<\gamma$ and $\alpha\leq\beta\leq\gamma$.
Essentially by considering the graph of $F(H)$, we show that $m(H)=0$ has a root $H^*>\frac{a}{3}=-\frac{\alpha}{3}$ if{}f $S>0$.
We have
\begin{align*}m(-\frac{\alpha}{3})&=(\alpha+\beta+\gamma)\frac{\alpha^2}{3}-\{2\alpha(\beta+\gamma)+2\beta\gamma-9\}\frac{\alpha}{3}
+\alpha\beta\gamma-(\alpha+\beta+\gamma)\\
&=\frac{\alpha}{3}\{\alpha^2-\alpha(\beta+\gamma)+\beta\gamma\}+2\alpha-\beta-\gamma\\
&=\frac{\alpha}{3}(\alpha-\beta)(\alpha-\gamma)+(\alpha-\beta)+(\alpha-\gamma)\text{ ,}
\end{align*}
with similar results for $m(-\frac{\beta}{3})$ and $m(-\frac{\gamma}{3})
$,
whence $m(-\frac{\alpha}{3})<0$ and $m(-\frac{\gamma}{3})>0$.
First we show that in all cases $m(H)=0$ has a root $h^*$ satisfying $-\frac{\gamma}{3}<h^*<-\frac{\alpha}{3}$.
\begin{equation}\label{D2}\text{As}\quad H\rightarrow -\frac{\gamma}{3}- \quad\text{and as}\quad H\rightarrow -\frac{\alpha}{3}+\text{,}
\qquad F(H)\rightarrow\infty\qquad\qquad\qquad\end{equation}
Then if $\beta=\gamma$, as $H\rightarrow-\frac{\gamma}{3}+$, $F(H)\rightarrow\infty$;
as $H\rightarrow-\frac{\alpha}{3}-$, $F(H)\rightarrow-\infty$;
and $F(H)$ is continuous over the interval $(-\frac{\gamma}{3},-\frac{\alpha}{3})$,
so there exists $h^*\in (-\frac{\gamma}{3},-\frac{\alpha}{3})$ with $F(h^*)=1$ and $m(h^*)=0$ as required;
and similarly if $\beta=\alpha$.
Also if $\alpha<\beta<\gamma$ and (a) $m(-\frac{\beta}{3})\neq 0$, then as $H\rightarrow-\frac{\gamma}{3}+$
and as $H\rightarrow-\frac{\alpha}{3}-$, $F(H)\rightarrow-\infty$, and $F(H)$ changes sign as $H$ increases
through the singularity at $H=-\frac{\beta}{3}$, so again there exists $h^*$ with $F(h^*)=1$ and $m(h^*)=0$ in
at least one of the intervals $(-\frac{\gamma}{3},-\frac{\beta}{3})$ and $(-\frac{\beta}{3},-\frac{\alpha}{3})$,
and so in $(-\frac{\gamma}{3},-\frac{\alpha}{3})$; while (b) if $m(-\frac{\beta}{3})=0$ then $h^*=-\frac{\beta}{3}$.
Hence $m(H)$ has at least one zero $h^*\in(-\frac{\gamma}{3},-\frac{\alpha}{3})$, as required.
Then, if $S=0$, $m(H)$ is linear and $h^*$ is the only root.
But, if $S>0$, $m(H)$ is quadratic and as $H\rightarrow\infty$, $F(H)\rightarrow 1-$, so with (\ref{D2})
this gives $H^*>-\frac{\alpha}{3}$ such that $F(H^*)=1$ and thus a unique second root $H^*$ of $m(H)=0$
with $H^*>\frac{a}{3}$.
Similarly for $S<0$, the second root $H^*$ satisfies $H^*<-\frac{\gamma}{3}<0$, and hence $H^*>\frac{a}{3}$
if{}f $S>0$, as required (in fact $H^*>\frac{a}{2}$, else $F(H^*)>1$, as is easily seen).

Assuming henceforth that $S>0$, from the above argument for $H>\frac{a}{3}$ we have $F(H)>F(H^*)=1$ if{}f $H<H^*$.
So if, with obvious notation, for $(\alpha_i,\beta_i,\gamma_i)$, $\alpha_1=\alpha_2=-a$, and for every $H>\frac{a}{3}$
we have $F_1(H)>F_2(H)$, then $F_1(H_2^*)>F_2(H_2^*)=1=F_1(H_1^*)$, so $H_2^*<H_1^*$. In particular we deduce that

(i) if $\alpha_1=\alpha_2$, $\beta_1\leq\beta_2$, $\gamma_1\leq\gamma_2$, with at least one strict inequality,
then $H_2^*<H_1^*$;

(ii) if $\alpha_1=\alpha_2=-a$ and $S_1=S_2=S$ (so $\sigma_1=\sigma_2=\sigma=\frac{a+S}{2}$),
but $\tau_1>\tau_2$, then $H_2^*<H_1^*$, i.e for fixed $a$, $S$ and $\sigma$, $H^*$ increases as $\tau$ increases.
This follows because $(\alpha,\beta,\gamma)=(\alpha,\sigma-\tau,\sigma+\tau)$ and
\begin{equation*}F(H)=H\sum\frac{1}{A}+\sum\frac{1}{AB}=\frac{H}{A}+\frac{1}{BC}\Bigl\{\Bigl(H+\frac{1}{A}\Bigr)\Bigl(B+C\Bigr)+1\Bigr\}
\text{ ;}\end{equation*}
 but $B+C=6H+2\sigma$ and $BC=(3H+\sigma)^2-\tau^2$; so for $H>\frac{a}{3}$, $F_1(H)>F_2(H)$ and $H_2^*<H_1^*$.

We write $H^\dagger(a,S,\tau):=H^*(\alpha,\beta,\gamma):=H^*$, and $h(\beta,\gamma):=H^*(-1,\beta,\gamma)$.
We observe that for a big$A$-$K_3N$, $H=H^*$, and that $A$, $B$, $C$ pairwise coprime and $H$ even requires that no two
of $\alpha, \beta, \gamma$ are even; so, since $\beta+\gamma=S+a$, $S$ odd requires $\alpha, \beta, \gamma$ all odd.
We now show that for any big$A$-$K_3N$, $H<2a^2$, considering cases~(a)~$a=1$, (b) $a\geq 2$, $S=1$, and (c) $a\geq 2$, $S\geq 2$.

(a) Any $K_3N$ with $a=1$ is a big$A$-$K_3N$.
So $\alpha=-1$, and for $\beta=0$, since $S>0$ and $B$, $C$ are coprime, $\gamma\geq 5$ and from (\ref{D1a})
and (i) above we have\newline
$h(0,\gamma)\leq h(0,5)\bumpeq 1.68<2\leq H$.
Also for $\beta\geq 2$, by (i) we have\newline
\hbox{$h(\beta,\gamma)<h(1,\gamma)\leq h(1,3)\bumpeq 1.63<2$,}
 covering all cases except \hbox{$h(1,2)\bumpeq 2.14\not\in 2\mathbb{N}$.}
Thus there are no $K_3N$'s with $a=1$.

(b) We have $S=1$, odd, so $a$ is odd and $a\geq 3$. Then maximum $\tau$ for given $a$
 occurs when $(\alpha,\beta,\gamma)=(-a,-a+2,2a-1)$, giving $\tau=\frac{3a-3}{2}$;
 and from (ii) above $H\leq g:=g(a):=H^\dagger(a,1,\frac{3a-3}{2})$,
 which from (\ref{D1a}) is given by
 \begin{multline*}3g^2-(6a^2-8a+13)g+2a^3-5a^2+2a-1=0,\quad\text{whence}\\
 g=\frac{1}{6}(6a^2-8a+13+\sqrt{36a^4-120a^3+280a^2-244a+181}),
 \end{multline*}
\vskip-12pt
i.e.
\begin{equation}\label{D3}H\leq g(a)=\frac{1}{6}(6a^2-8a+13+\sqrt{(6a^2-10a+15)^2+68a-44})\\
\end{equation}$\qquad\qquad\qquad\qquad<2a^2\quad\text{for}\quad a\geq 3.$

(c) First we find $H^\dagger(a,S,a+\sigma)=H^*(-a,-a,2a+S)$, which from (\ref{D1b}) is the root $H^*$ of
\hbox{$3SH^2-(2a^2+8a\sigma+9)H+a^3+2a^2\sigma-2\sigma+a=0$,} which has 
\begin{align*}\text{discriminant}&=
(2a^2+8a\sigma+9)^2-12(2\sigma-a)(a^3+2a^2\sigma-2\sigma+a)\\
&=16a^4+32a^3\sigma+16a^2\sigma^2+48a^2+96a\sigma+48\sigma^2+81\\
&=16(a^2+3)(a+\sigma)^2+81\qquad\text{and hence}
\end{align*}
\vskip-12pt
\begin{subequations}\label{D4}
\begin{align}
H^\dagger(a,S,a+\sigma)&=\frac{1}{6S}(2a^2+8a\sigma+9+\sqrt{16(a^2+3)(a+\sigma)^2+81})\\
&=\frac{1}{6S}(2a^2+4a(a+S)+9+\sqrt{4(a^2+3)(3a+S)^2+81})\label{D4a}\\
&=\frac{a^2}{S}+\frac{2a}{3}+\frac{3}{2S}+\sqrt{\Bigl(a^2+3\Bigr)\Bigl(\frac{a}{S}+\frac{1}{3}\Bigr)^2+\frac{9}{4S^2}}\label{D4b}
\end{align}
\end{subequations}

So from (i) or (\ref{D4b}) for given $a$, $H^\dagger(a,S,a+\sigma)$ decreases as $S$ increases,
and hence for any big$A$-$K_3N$ with $S\geq 2$ and $a\geq 2$, also using (ii) we have
\begin{equation*}H=H^\dagger(a,S,\tau)<H^\dagger(a,S,a+\sigma)\leq H^\dagger\Bigl(a,2,\frac{3a+2}{2}\Bigr)
\text{,}\end{equation*}
and from~(\ref{D4a}) we have
\begin{align*}H^\dagger\Bigl(a,2,&\frac{3a+2}{2}\Bigr)<2a^2\Leftrightarrow
\sqrt{\Bigl(a^2+3\Bigr)\Bigl(\frac{a}{2}+\frac{1}{3}\Bigr)^2+\frac{9}{16}}<\frac{3}{2}a^2-\frac{2a}{3}-\frac{3}{4}\\
&\Leftrightarrow\frac{a^4}{4}+\frac{a^3}{3}+\Bigl(\frac{3}{4}+\frac{1}{9}\Bigr)a^2+a+\frac{1}{3}
<\frac{9}{4}a^4-2a^3+\Bigl(\frac{4}{9}-\frac{9}{4}\Bigr)a^2+a\\
&\Leftrightarrow 6a^4-7a^3-8a^2>1\\
&\Leftrightarrow a^3(4a-7)+2a^2(a^2-4)>1\text{, which is true for $a\geq 2$.}
\end{align*}

Thus for all big$A$-$K_3N$'s $H<2a^2$, and Theorem \ref{Th4_1} follows.
\end{proof}

It is easily shown that $H\leq g(a)<2a^2$ in case (c) as well as case (b)
and, for $H\geq 6,\quad g^{-1}(H)>\sqrt{\frac{H}{2}}$; hence for all $K_3N$'s
\begin{equation}\label{D5}A\leq3H-g^{-1}(H)\text{,}\end {equation}
which for $H\geq 6$ is
a slightly stronger but less convenient result than Theorem~\ref{Th4_1}.

It is also not onerous to extend the approach of case (a) for finding all $K_3N$'s with $a=1$:
for $a=2$, again there are none (although $9801=9 \cdot 11 \cdot 99$ obeys the Korselt divisibility criterion);
but for $a=3$ there are two, the big$A$-$C_3N$ $7 \cdot 23 \cdot 41=6601,$ and the big$A$-$K_3N$
\begin{equation}\label{D6}n^*=547 \cdot 575 \cdot 659=207271975.\end{equation}
Actually from (\ref{D3}), $g(3)=14$, giving rise to $n^*$ with $(A,B,C)=(39,41,47)$
and $(\alpha,\beta,\gamma)=(-3,-1,5)$; and $39=A<3H-\sqrt{\frac{H}{2}}\bumpeq 39.35$.
$n^*$ is the only $K_3N$ with equality in (\ref{D5}), since $g(5)$ is irrational,
and for $a\geq 7$ we have
\begin{equation*}6a^2-10a+15<\sqrt{(6a^2-10a+15)^2+68a-44}<6a^2-10a+16\text{,}\end{equation*}
so from (\ref{D3}) $g(a)$ is irrational.
Gordon Davies (see~\S\ref{Sec5a})
 did a computer search for \hbox{big$A$-$K_3N$'s} with $S=1$,
using $(\alpha,\beta,\gamma)=(-a,-a+2t,2a-2t+1)$, for odd $a$ up to 1239
and $1<t<\frac{3a}{4}$: no more were found, and $H^*$ was rational only for $a=151$, $t=89$, giving $H^*=13067\frac{1}{3}$.

In like manner to (\ref{D4}) we find
\begin{align*}H^\dagger(a,S,0)&=\frac{1}{6S}(4a\sigma-2\sigma^2+9+\sqrt{4(\sigma^2+3)(a+\sigma)^2+81}\ )\\
&=\frac{1}{12S}\{(a+S)(3a-S)+18+\sqrt{[(a+S)^2+12](3a+S)^2+324}\}\text{,}
\end{align*}
so for fixed $S$ and large $a$ we have $H^\dagger(a,S,0)\sim \frac{a^2}{2S}$ and $H^\dagger(a,S,a+\sigma)\sim\frac{2a^2}{S}$,
whence $(1+o(1))\frac{a^2}{2S}<H^*<(1+o(1))\frac{2a^2}{S}$.
Hence with $\lambda:=\frac{a^2}{H}$ and $A=3H-\sqrt{\lambda H}$, for $a\gg S$ we expect $\frac{S}{2}<\lambda<2S$;
we may regard $\frac{A}{H}<3$ and $\lambda>\frac{1}{2}$ as different measures of closeness to the bound of Theorem \ref{Th4_1}.
For $n<10^{24}$, there are only 71 big$A$-$C_3N$'s, and just 11 with $\frac{A}{H}>2$;
the two largest $\frac{A}{H}$ values are about 2.342 and 2.683, and only these two big$A$-$C_3N$'s have $a>S$.
With $\frac{A}{H}\bumpeq 2.683$, we have $n^\dagger:=835327 \cdot 893359 \cdot 1117117=833645090806507981$,
with $(A,B,C,H)=(1497,1601,2002,558)$, so $a=177$, $S=78$ and $\lambda\bumpeq 56.15$.
For $n^*$, from (\ref{D6}), $\frac{A}{H}\bumpeq2.786$ and $\lambda=\frac{9}{14}$.
I was able to find marginal improvements on the bound of Theorem~\ref{Th4_1} for sufficiently large $a$
(e.g. (\ref{D7}) below), but none implying $\lambda>\mu$ for some $\mu>\frac{1}{2}$.

\begin{tabbing}
Challenges \=1(a): Find a big$A$-$K_3N$ with $a>3$ and $S=1$.\\
\>1(b): \=Find a big$A$-$C_3N$ with $n>6601$ and $\lambda<10$.
\end{tabbing}

\begin{theorem}\label{Th4_2}For any $K_3N$, $A<\sqrt{3(p-1)}-\frac{1}{2}\sqrt[4]{\frac{p-1}{12}}$.
\end{theorem}
\begin{proof}Using the notation and from the above discussion of Theorem \ref{Th4_1}, \hbox{Theorem \ref{Th4_2}}
holds for $n^*$ and for $n=6601$, and hence for $a\leq 3$. For $S\geq 2$ and $a\geq 4$, from~(\ref{D4a})
we have
\begin{align*}H<H^\dagger(a,2,\frac{3a+2}{2})=&\frac{1}{12}\bigl(6a^2+8a+9+\sqrt{4(a^2+3)(3a+2)^2+81}\bigr)\\
<&\frac{1}{12}\bigl(6a^2+8a+9+\sqrt{4(a+1)^2(3a+2)^2}\bigr)\\
=&a^2+\frac{3}{2}a+\frac{13}{12}<2\Bigl(a-\frac{1}{2}\Bigr)^2
\quad\text{since  }a\geq 4\text{ ;}
\end{align*}
and for $S=1$, $a$ is odd, and for $a\geq 5$ from (\ref{D3}), $H\leq g(a)<2(a-\frac{1}{2})^2$.
Hence for $a\geq 4$, for any $K_3N$, $H<2(a-\frac{1}{2})^2$, so $a>\sqrt{\frac{H}{2}}+\frac{1}{2}$
and
\begin{equation}\label{D7}A<3H-\sqrt{\frac{H}{2}}-\frac{1}{2}\end{equation}
Let $f(x):=\dfrac{3p\,'}{x}-x-\sqrt{\dfrac{p\,'}{2x}}-\dfrac{1}{2}.$
Then since $H=\dfrac{p\,'}{A}$, from (\ref{D7}) $f(A)>0$.
Also
\begin{equation*}f'(x)=-\frac{3p\,'}{x^2}-1+\frac{1}{2}\sqrt{\frac{p\,'}{2x^3}}
=-\frac{\sqrt{p\,'}}{2\sqrt{2}x^2}(6\sqrt{2p\,'}-\sqrt{x})-1<0\quad\text{for  }0<x<p\,',
\end{equation*}
so $f(x)$ decreases as $x$ increases over this interval, which certainly contains $A$.\newline
Writing $\quad\rho:=\sqrt[4]{3p\,'}\quad\text{and}\quad \mu:=\sqrt{\rho^2-\frac{\rho}{2\sqrt{6}}}
\quad$we have
\begin{align*}f\Bigl(\sqrt{3p\,'}&-\frac{1}{2}\sqrt[4]{\frac{p\,'}{12}}\Bigr)=f(\mu^2)
=\frac{\rho^4}{\mu^2}-\mu^2-\frac{\rho^2}{\sqrt{6}\mu}-\frac{1}{2}\\
&=\frac{1}{\sqrt{6}\mu^2}\bigl\{\sqrt{6}(\rho^4-\mu^4)-\rho^2\mu\bigr\}-\frac{1}{2}
=\frac{1}{\sqrt{6}\mu^2}\Bigl(\rho^3-\frac{\sqrt{6}}{24}\rho^2-\rho^2\mu\Bigr)-\frac{1}{2}\\
&=\frac{\rho^2\{(\rho-\frac{\sqrt{6}}{24})^2-\mu^2\}}{\sqrt{6}\mu^2(\rho-\frac{\sqrt{6}}{24}+\mu)}-\frac{1}{2}
=\frac{\rho^2}{96\sqrt{6}\mu^2(\rho+\mu-\frac{\sqrt{6}}{24})}-\frac{1}{2}
\end{align*}
\begin{equation*}\text{But }p\geq 3,\text{ so }\rho>1.565, \mu>1.459, \;\frac{\rho^2}{\mu^2}<1.15,
\;\frac{\rho^2}{96\sqrt{6}\mu^2(\rho+\mu-\frac{\sqrt{6}}{24})}<0.00168
\end{equation*}
and $f(\mu^2)<0$, whence since $f(A)>0$ and $f(x)$ decreases with $x$, $A<\mu^2$, giving Theorem \ref{Th4_2}.
\end{proof}

For $n^*$ (see (\ref{D6})), $A=39<\sqrt{3p\,'}-\frac{1}{2}\sqrt[4]{\frac{p\,'}{12}}\bumpeq 39.1736$.

In seeking an upper bound for $B$ given $p$, we know from Theorem \ref{Th3_1} that we can have
$q\,'=p\,'(2p-\sqrt{p-\frac{3}{4}}+\frac{1}{2})$, with $H=p\,'$ and $B=2p-\sqrt{p-\frac{3}{4}}+\frac{1}{2}$;
we~show that it is posssible for $B$ marginally to exceed this:

\begin{theorem}\label{Th4_3}(a) For any $K_3N$, $B<2p-\sqrt{p-\frac{3}{4}}+\frac{\sqrt{3}+1}{2\sqrt{3}}$

$\qquad\qquad\qquad$(b) For any $C_3N$, $B<2p-\sqrt{p-\frac{3}{4}}+\frac{\sqrt{7}+1}{2\sqrt{7}}$
\end{theorem}
\begin{proof}Write $S=\sqrt{p-\frac{3}{4}}-\frac{1}{2}$. We seek $B>2p-S$, and we consider two cases:

(i) $G\geq 2$. Then from Theorem \ref{Th2_2}, $E\leq p-1$, and from (\ref{B19})\newline
\begin{equation*}B=\frac{p+E}{G}\leq\frac{2p-1}{2}<p<2p-S\end{equation*}

(ii) $G=1$. Then from (\ref{B19}) $B=p+E$ and from (\ref{B10}) $E=p-s$, so
\begin{equation}\label{D8}B=2p-s\end{equation}
Also from (\ref{B20}), $\Delta=AG=A$, whence from (\ref{B11})
\begin{equation}\label{D9}\eta=\frac{s^2+A}{p-s}\end{equation}
and using (\ref{B9b}) and (\ref{B10}), $CH=r'=\dfrac{AH(p+D)}{A},\quad $ so
\begin{equation}\label{D10}C=p+D=2p+s+\eta\end{equation}
Since $H$ is even, from (\ref{B18})
\begin{equation}\label{D11}H=AF-1
\quad\text{and}\; A\;\text{and}\;F\;\text{are both odd, and also}\end{equation}
\begin{equation}\label{D12}p=AH+1=FA^2-A+1.\end{equation}

We now consider the two sub-cases (a) $\eta\geq 2$ and (b) $\eta=1$.

(a) If $\eta\geq 2$, from (\ref{D9}) we have $s^2+A\geq 2(p-s)$, whence $s^2+2s\geq 2p-A$\newline
and $s\geq \sqrt{2p-A+1}-1>\sqrt{2p-\sqrt{3p-3}+1}-1$, using Theorem \ref{Th4_2}.\newline
But $\sqrt{2p-\sqrt{3p\,'}+1}-1>S$ reduces to $p+\frac{3}{2}>\sqrt{p-\frac{3}{4}}+\sqrt{3p\,'}$,
which is true for $p\geq 3$, so $s>S$ and from (\ref{D8}) $B<2p-S$.

(b) If $\eta=1$, from~(\ref{D9}) $s^2+A=p-s$, whence from~(\ref{D12})
\begin{subequations}\label{D13}
\begin{equation}\label{D13a}s^2+s=s(s+1)=FA^2-2A+1=F'A^2+A'\,^2\quad\text{and}\end{equation}
\begin{equation}\label{D13b}s=\sqrt{FA^2-2A+\frac{5}{4}}-\frac{1}{2}\end{equation}
\end{subequations}
From~(\ref{D13a}) we note that $F=1$ implies $A'\,^2=s(s+1)$, which is impossible.
So by (\ref{D11}) for odd $F\geq 3$, remembering that $x=x'+1$ and using (\ref{D12}) again, we have
\begin{equation*}S-s=\sqrt{FA^2-A+\frac{1}{4}}-\sqrt{FA^2-2A+\frac{5}{4}}
=\frac{A'}{\sqrt{(F'A^2+(A-\frac{1}{2})^2}+\sqrt{F'A^2+A'\,^2+\frac{1}{4}}}
\end{equation*}
Hence if $A=1$, $s=S$ as for Theorem \ref{Th3_1}, but if $A\geq 3$ ($A$ is odd by (\ref{D11})) then
\begin{equation*}S-s=\frac{1}{\sqrt{F'(1+\frac{1}{A'})^2+(1+\frac{1}{2A'})^2}
+\sqrt{F'(1+\frac{1}{A'})^2+1+\frac{1}{4A'\,^2}}}>0
\end{equation*}
From this we deduce that $0<S-s<\frac{1}{2\sqrt{F}}\leq\frac{1}{2\sqrt{3}}$ since $F\geq 3$,\newline
whence $2p-S<B=2p-s<2p-S+\frac{1}{2\sqrt{3}}$, giving Theorem \ref{Th4_3}(a).

Also we see that for given $F$, if there is an infinite sequence of $K_3N$'s with increasing $A$ values,
then $2p-S+\frac{1}{2\sqrt{F}}-B\rightarrow0+$ as $A\rightarrow\infty$.
There is such a sequence if{}f (\ref{D13a}) has an infinite number of solutions for $s$ and $A$ which result
in pairwise coprime $A$, $B$ and $C$, and $E\geq 2$. Such solutions we call {\emph {acceptable}}.
With 
\begin{equation}\label{D14}\phi:=2FA-2=2H,\quad \theta=2s+1,\qquad\text{(\ref{D13a}) implies}\end{equation}
\begin{equation}\label{D15}\phi^2-F\theta^2=4-5F,\end{equation}
and from the theory of quadratic forms a necessary condition for this to have a solution is that
$x^2\equiv F\pmod{(5F-4)}$ be soluble.

If $F\equiv 0\pmod 3$, suppose $(\phi,\theta)$ gives an acceptable solution. 
Then {\emph {working in }}$\mathbb{Z}_3$, from (\ref{D14}) $\theta=2s+1$, so $s=2\theta+1$,
from (\ref{D13a}) $s(s+1)=FA^2-2A+1$,
so $A=(2\theta+1)(2\theta+2)-1=\theta^2+1$, from (\ref{D11}) $H=FA-1=2$; so $p=AH+1=2\theta^2=2$ unless $\theta=0$, $p=0$;
from (\ref{D8}) $B=2p-s=\theta^2-2\theta-1=(\theta-1)^2+1$,
so $q=BH+1=2(\theta-1)^2=2$ unless $\theta=1$, $q=0$; with $\eta=1$, from (\ref{D8},~\ref{D10},~\ref{D14})
$C=B+\theta=\theta^2-\theta+2$, so $r=CH+1=2\theta(\theta-1)+2=2$ unless $\theta=2$, $r=0$.
Thus in $\mathbb{Z}_3$, for any $\theta$ exactly one of $p,q,r$ is zero, whence in $\mathbb{Z}$
$3|p,q\text{ or }r$ and so $n$ is only a $C_3N$, possibly, if $p=3$, in which case $n=561$, with $A=1$.

For $F=5$, (\ref{D15}) is $\phi^2-5\theta^2=-21$, but $x^2\equiv 5\pmod{21}$ has no solutions,
so neither has (\ref{D15}). Thus there are no $C_3N$'s with $A>G=\eta=1$ and $F<7$, and
Theorem~\ref{Th4_3}(b) follows in like manner to Theorem~\ref{Th4_3}(a) above.
\end{proof}

From (\ref{D14}) for a solution to (\ref{D15}) to yield a solution to (\ref{D13a}) 
we need $\phi\equiv-2\pmod{2F}$ and $\theta$ odd.
Clearly from (\ref{B10}) and (\ref{D9}) with $\eta=1$, $E=p-s=s^2+A\geq 2$, as required.
Also such a solution will result in positive integers $A, B, C, H, F$ which satisfy (\ref{B5a}),
whence we have $h:=\gcd(A,B,C)=\gcd(A,B)=\gcd(A,C)=\gcd(B,C)$;
and from (\ref{D8}) and (\ref{D12}) $s\equiv2\pmod h$, whence from (\ref{D13a}) $6\equiv 1\pmod h$,
so $h=1\text{ or }5$.
Then if $h=5$, from (\ref{D14}) we have $(\phi,\theta)\equiv (3,0)\pmod 5$;
so if $(\phi,\theta)\not\equiv (3,0)\pmod 5$ then $h=1$
 and our solution is acceptable.

Using the theory of Pell's equation, if $\phi_i\,^2-F\theta_i\,^2=k$, $x^2-Fy^2=1$ and
\begin{equation}\label{D16}\phi_{i+1}=(2Fy^2+1)\phi_i+2Fxy\theta_i\,,\qquad \theta_{i+1}=2xy\phi_i+(2Fy^2+1)\theta_i ,\end{equation}
then $\phi_{i+1}\,^2-F\theta_{i+1}\,^2=k$ and $\phi_{i+1}\equiv \phi_i\pmod{2F}$; so (\ref{D15}) has an infinite sequence
of acceptable solutions if it has a solution $(\phi_1,\theta_1)$ with
$\phi_1\equiv-2\pmod {2F}$ (unless, if possible, $(\phi_1,\theta_1)\equiv(3,0)$ and $y\equiv 0\pmod 5)$.
If $F=s^2+s+1$, as for the $K_3$-family $\{n_2(s)\}$ (equations (\ref{C4})) then $(\phi_1,\theta_1)=(2s(s+1),2s+1)$
is acceptable and gives rise to $n_2(s)$ with $A=1$ and $B=2p-S$.
Also there are acceptable solutions with $F\neq s^2+s+1$: for example for $F=87$, $(\phi_1,\theta_1)=(1912,205)$ and $A=11$.

For $F=3$, $s=1$, so $(\phi_1,\theta_1)=(4,3)$, which gives $n_2(1)=561$; and for $F=7$, $s=2$, $(\phi_1,\theta_1)=(12,5)$,
(\ref{D16}) is $\phi_{i+1}=127\phi_i+336\theta_i ,\ \theta_{i+1}=48\phi_i+127\theta_i$, and \hbox{$(\phi_2,\theta_2)=(3204,1211)$} leads
via~(\ref{D14}, \ref{D11}, \ref{D8}, \ref{D10}) to
$A=229$, $s=605$, \hbox{$H=1602$,} $p=366859$, $B=733113$, $q=1174447027$, $C=734324$, $r=1176387049$ and
$2p-S\bumpeq 733112.8118<B=733113<2p-S+\frac{1}{2\sqrt{7}}\bumpeq 733113.000737$ for the $K_3N$ $n=pqr$.

If $n(i)$ is the $K_3N$ arising from $(\phi_i,\theta_i)$ for $F=7$, it is easily shown that, for \hbox{$i\geq 2$,}
$n(i+1)\bumpeq 254^8\,n(i)\bumpeq 1.73\times 10^{19}\,n(i)$, and a naive ``probability'' estimate based
on the knowledge that $n(2)$ is not a $C_3N$ and an assumption of the independence
of the primality of $p,q\text{ and }r$ is :``$p$''$(n(i)$ is a $C_3N$ for some $i>2$)$\bumpeq\frac{1}{2000}$.

\begin{tabbing}
Challenges \=2(a): Find a $C_3N$ with $B>2p-\sqrt{p-\frac{3}{4}}+\frac{1}{2}+\frac{1}{4\sqrt{3}}$.\\
\>2(b): Find a $C_3N$ with $B>2p-\sqrt{p-\frac{3}{4}}+\frac{1}{2}$.
\end{tabbing}

From Theorem \ref{Th3_2}, $r'=\frac{1}{2}p\,'(p+1)^2$ is possible, so given $p$ we can have
\newline
$C=\frac{1}{2}(p+1)^2$.
Again, we show that this can be slightly exceeded:

\begin{theorem}\label{Th4_4}For any $K_3N$, $C\leq\frac{1}{2}(p\,^2+2p+\frac{1}{2}\sqrt{4p-3}+\frac{1}{2})$,
with equality if{}f $E=2$, $F=G=1$.
\end{theorem}
\begin{proof}For any $K_3N$, $A<p$, and so from (\ref{B16}) and (\ref{B19}) we have
\begin{equation}\label{D17}C=\frac{p\,^2}{EG}+\frac{p}{G}+\frac{A}{E}<\frac{p\,^2}{EG}+\frac{p}{G}+\frac{p}{E}
\end{equation}
But $E\geq 2$ and $G\geq 1$, so if $E\geq 3$ then
$C<\frac{p\,^2}{3}+p+\frac{p}{3}=\frac{p(p+4)}{3}<\frac{1}{2}(p+1)^2$ for $p\geq 3$, and if $G\geq 2$ then
$C<\frac{p\,^2}{4}+\frac{p}{2}+\frac{p}{2}=\frac{p(p+4)}{4}<\frac{1}{2}(p+1)^2$;
thus $C<\frac{1}{2}(p+1)^2$ unless $E=2, G=1$ in which case, since $G=1$, from (\ref{D12}) $FA^2-A-p\,'=0$,
\begin{equation*}\text{whence  }A=\frac{1+\sqrt{4Fp\,'+1}}{2F}
=\frac{1}{2F}+\frac{1}{2\sqrt{F}}\sqrt{4p\,'+\frac{1}{F}}\leq\frac{1}{2}(1+\sqrt{4p-3})\text{, with equality}\end{equation*}
if{}f  $F=1$.
Hence from (\ref{D17}), $C\leq\frac{p\,^2}{2}+p+\frac{1}{4}(1+\sqrt{4p-3}),$ and Theorem~\ref{Th4_4} follows.
\end{proof}

For a $K_3$-family with equality in Theorem~\ref{Th4_4}, from (\ref{B18}), (\ref{B19}) and (\ref{B16}) we get
\begin{multline}\label{D18}A=A(H):=H+1,\quad B=AH+3=B(H):=H^2+H+3, \\C=C(H):=\frac{1}{2}\Bigl(H^4+2H^3+5H^2+5H+4\Bigr)\end{multline}
\begin{multline*}p=P(H):=H^2+H+1, \quad q=Q(H):=H^3+H^2+3H+1, \\r=R(H):=\frac{1}{2}\Bigl(H^5+2H^4+5H^3+5H^2+4H+2\Bigr)\quad\text{and}
\end{multline*}
\begin{equation}\label{D19}n=N(H):=\frac{1}{2}\Bigl(H^{10}+4H^9+14H^8+30H^7+53H^6+69H^5+71H^4+55H^3+31H^2+12H+2\Bigr)\end{equation}
With $H=2h$ this yields a $K_3$-family with parameter $h$, but from (\ref{D18}) if $H\equiv 2\pmod 3$
then $\gcd(A,B)=3$, so we take $H=6t$ or $H=6t+4$ to get two \hbox{$K_3$-families} with the required maximal $C$ property.
Also if $H\equiv 1 \pmod 3$ then $3|p$ and $3|q$, so for $C_3N$'s we must have the $K_3$-family with $H=6t$.
A search by Matthew Williams (see~\S\ref{Sec5a}) up to $t=1365$ found $C_3N$'s only for $t=1$ and \hbox{$t=210$,}
 giving $n=43 \cdot 271 \cdot 5827=67902031$, with $A=7, B=45, C=971$,
and \hbox{$n=1588861 \cdot 2001967381 \cdot 1590423947471521=5058896665381789187674264635361$,} with $A=1261, B=1588863, C=1262241228152$.

The above $K_3$-families will be shown to exemplify further bounds in Theorems \ref{D9} and \ref{D10}.
\subsection{An upper bound for $q$, given $p$ and $n$}\label{Sec4b}
\begin{theorem}\label{Th4_5}
For any $K_3N$, $q<\sqrt{\dfrac{n}{p}}-\sqrt{\dfrac{p}{12}}$.
\end{theorem}
\begin{proof}Since $r\geq q+H,$\ we have $q=\dfrac{n}{pr}\leq\dfrac{n}{p(q+H)},$
whence $pq^2+pHq-n\leq 0$
and $q\leq\sqrt{\dfrac{n}{p}+\dfrac{H^2}{4}}-\dfrac{H}{2}$.
\qquad Now \(\sqrt{\dfrac{n}{p}+\dfrac{H^2}{4}}-\dfrac{H}{2}<\sqrt{\dfrac{n}{p}}-\sqrt{\dfrac{p}{12}}\)
\newline
if{}f\quad\(\sqrt{\dfrac{n}{p}+\dfrac{H^2}{4}}-\sqrt{\dfrac{n}{p}}<\dfrac{H}{2}-\sqrt{\dfrac{p}{12}}
=\dfrac{1}{2}\Bigl(H-\sqrt{\dfrac{AH+1}{3}}\Bigr)\);
\newline
also\quad \(\sqrt{\dfrac{n}{p}+\dfrac{H^2}{4}}-\sqrt{\dfrac{n}{p}}
=\dfrac{H^2}{4\Bigl(\sqrt{qr+\frac{H^2}{4}}+\sqrt{qr}\Bigr)}<\dfrac{H^2}{8\sqrt{qr}}<\dfrac{H^2}{8\sqrt{q\,'r'}}
=\dfrac{H}{8\sqrt{BC}},\)
so~Theorem~\ref{Th4_5} follows if we can show that
\(\dfrac{H}{8\sqrt{BC}}<\dfrac{1}{2}\Bigl(H-\sqrt{\dfrac{AH+1}{3}}\Bigr),\)
i.e.
\begin{equation}\label{D20}1<4\sqrt{BC}
\Bigl(1-\sqrt{\frac{A}{3H}+\frac{1}{3H^2}}\Bigr).\end{equation}
We consider two cases: (i) if $\frac{A}{H}<\frac{13}{6}$, then since $B\geq 2, C\geq 3\text{ and }H\geq 2$,
\begin{equation*}\Bigl(1-\frac{1}{4\sqrt{BC}}\Bigr)^2-\frac{1}{3H^2} \geq \Bigl(1-\frac{1}{4\sqrt{6}}\Bigr)^2-\frac{1}{12}
>\frac{13}{18}>\frac{A}{3H},\quad\text{whence (\ref{D20}) follows.}
\end{equation*}
(ii) If $\frac{A}{H}\geq \frac{13}{6}$, then $n$ is a big$A$-$K_3N$, $B>A\geq \frac{13H}{6}$ and $C=3H+\gamma>3H$
(since $\gamma>0$, see the proof of Theorem~\ref{Th4_1}, just after (\ref{D1})).
From Theorem~\ref{Th4_1},
\begin{multline*}A<3H-\sqrt{\frac{H}{2}},\quad\text{so}\quad \frac{A}{3H}+\frac{1}{3H^2}<1-\frac{1}{3\sqrt{2H}}+\frac{1}{3H^2}.
\quad\text{Now}\quad 1-(1-x)^{\frac{1}{2}}>\frac{1}{2}x
\\\text{for } 0<x<1.\text{ Hence }
4\sqrt{BC}\Bigl(1-\sqrt{\frac{A}{3H}+\frac{1}{3H^2}}\Bigr)
>4\sqrt{BC}\Bigl\{1-\Bigl[1-\Bigl(\frac{1}{3\sqrt{2H}}-\frac{1}{3H^2}\Bigr)\Bigr]^\frac{1}{2}\Bigr\}
\\>4\sqrt{\Bigl(\frac{13H}{6}\Bigr) \cdot 3H}\Bigl\{\frac{1}{2}\Bigl(\frac{1}{3\sqrt{2H}}-\frac{1}{3H^2}\Bigr)\Bigr\}=
\frac{\sqrt{26}}{3}\Bigl(\sqrt{\frac{H}{2}}-\frac{1}{H}\Bigr)>1\quad\text{for }\ H\geq 4,
\end{multline*}
which is (\ref{D20}); and for $H=2$, since $\frac{13H}{6}\leq A<3H-\sqrt{\frac{H}{2}}$,
we have $4\frac{1}{3}\leq A<5$, which is impossible, establishing the result.
\end{proof}
For the $C_3N$ $191 \cdot 421 \cdot 431$, Theorem~\ref{Th4_5} gives $q=421<421.981\bumpeq\sqrt{\frac{n}{p}}-\sqrt{\frac{p}{12}}$.

\subsection{Upper bounds given $n$ for $p, A, B, C$ and $ABC$}\label{Sec4c}
A cursory glance at a list of $C_3N$'s suggests that a substantially better bound than the $p<\sqrt[3]{n}$
given by Theorem~\ref{Th3_4} should be attainable.
That this is not so can be seen by considering the Chernick type $K_3$-families
discussed in \S\ref{Sec2b} with \hbox{$(A,B,C)=(2u-1,2u,2u+1)$.}
For this family 
$\sum A=6u, \sum AB=12u^2-1, ABC=2u(4u^2-1)$, so (\ref{B5a}) requires $H$ and $F$ satisfying 
$(12u^2-1)H+6u=2u(4u^2-1)F$. We see that for 
$u>1$, $F=6u^2-5$ and $H=H_o=4u(u^2-1)$ is the unique solution with $0<H_o<ABC$,
so the general solution is
\begin{equation}\label{D21}H=H_o+tABC=2u\{2(u^2-1)+(4u^2-1)t\}
=\frac{B}{2}\{B^2-4+2(B^2-1)t\}.\end{equation}
In terms of $B=2u$, this gives \begin{align*}p_f:=p_f(u,t):&=\frac{1}{2}B(B-1)\{B^2-4+2(B^2-1)t\}+1,\\
q_f:=q_f(u,t):&=\frac{1}{2}B^2\{B^2-4+2(B^2-1)t\}+1,\\
r_f:=r_f(u,t):&=\frac{1}{2}B(B+1)\{B^2-4+2(B^2-1)t\}+1,
\end{align*}
and we have the two parameter system of $K_3$-families $n_f(u,t):=p_fq_fr_f$ for $u=1, t\geq1$ and $u>1, t\geq 0.$
If now we arbitrarily describe any $KN$ with $\frac{r}{p}<1.5$ as ``flat'',
then $n_f(u,t)$ gives flat $K_3N$'s for $u\geq 3,$ since in general
$\frac{r}{p}=\frac{CH+1}{AH+1}=\frac{C}{A}-\frac{C-A}{Ap}<\frac{C}{A},$
giving $\frac{r}{p}=1+\frac{2}{B-1}-\frac{2}{(B-1)p}<1.4$ for $B\geq 6$;
this also shows that for any Chernick type $K_3$-family the first member is the
``flattest'', since \hbox{``steepness''$:=\frac{r}{p}$} increases with $p$ (and hence $t$).
In \cite{Pinch} Pinch says that for $CN$'s up to $10^{15}$ the largest value of $p$ occurring is
$72931$, dividing $651693055693681=72931 \cdot 87517 \cdot 102103$; this is $n_f(3,69)$.
Gordon Davies did a cursory computer search which found the $C_3N$'s $n_f(5,3), n_f(8,0), n_f(17,1), n_f(51,0)$
and the very flat $n_f(102,0)=861618073 \cdot 865862497 \cdot 870106921=649136982888522736355512801,$
with ``steepness''$\bumpeq 1.00985.$ Clearly any improvement on $p<\sqrt[3]{n}$ will be of the form
$p<(1-o(1))\sqrt[3]{n},$ with $o(1)>0$, and we now show that
\begin{theorem}\label{Th4_6}For any $K_3N$,
\begin{equation*}p\leq \Bigl\lceil\sqrt[3]{n}-\frac{4\sqrt{3}}{9}\sqrt[6]{n}\Bigr\rceil .\end{equation*}
\end{theorem}
\begin{proof}
If $k:=B-A$ and $l:=C-A$, for any $K_3N$ we have $n=p(p+kH)(p+l H).$
From Theorem~\ref{Th4_1}, $A<3H-1$ so 
\begin{equation}\label{D22}
 H^2>\frac{AH+H}{3}>\frac{AH+1}{3}=\frac{p}{3},
\end{equation}
 and hence
$kH\geq H>\sqrt{\frac {p}{3}}$ and for $l\geq 3$, $l H>3\sqrt{\frac{p}{3}}$,
while if $l=2$ we have $n\in\{n_f(u,t)\}$, and $l H=2H>3\sqrt{\frac{p}{3}}$
is equivalent to $4H^2>3p=3AH+3$, i.e. to $H\{4H-3(B-1)\}>3$, which is obvious from~(\ref{D21});
so in all cases $l H>3\sqrt{\frac{p}{3}}=\sqrt{3p}.$
So for any $K_3N$, $n>p(p+\sqrt{\frac{p}{3}})(p+\sqrt{3p})$, and if $n=V(V+\sqrt{\frac{V}{3}})(V+\sqrt{3V})$,
then $p<V$. 
Using the approach of Theorem~\ref{Th3_5} to express $V$ in terms of $n$ as a power series, if 
we put $n=27y^6$, $V=3v^2$, $v=y+\sum_{i=0}^{\,\infty}a_iy^{-i}$ and equate coefficients of $y^5$, $y^4$ and $y^3$,
we get $a_o=-\frac{2}{9}, a_1=\frac{11}{162}, a_2=-\frac{52}{2187}$
and $V=3y^2-\frac{4}{3}y+\frac{5}{9}-\frac{170}{729y}+\ldots $,
whence $V=\sqrt[3]{n}-\frac{4\sqrt{3}}{9}\sqrt[6]{n}+\frac{5}{9}-\frac{170\sqrt{3}}{729\sqrt[6]{n}}+\ldots$
But $p<V$  and $p$ is an integer, so Theorem~\ref{Th4_6} follows.
\end{proof}

From~(\ref{D6}) for $n^*$ we get $p=547<\lceil\sqrt[3]{n}-\frac{4\sqrt{3}}{9}\sqrt[6]{n}\,\rceil=574$,
and clearly a marginally lower bound than that given by Theorem~\ref{Th4_6} could readily be established.
I speculate that for large $n$ there is an upper bound $\sqrt[3]{n}-\mu\sqrt[4]{n}(1+o(1))$ for some $\mu>0$;
such a bound can be established for the system $n_f(u,t)$ with $\mu=\frac{1}{\sqrt[4]{2}}$.
But $n^*$ (see (\ref{D6})) is both slightly flatter and much smaller than $n_f(5,0)$: if (using the \hbox{notation}
of the proof of Theorem~\ref{Th4_1}) there were an infinite system of big$A$-$K_3N$'s with
$(\alpha,\beta,\gamma)=(-\sqrt{\frac{H}{2}},1,\sqrt{\frac{H}{2}})$, from (\ref{B6}) we can deduce that it would require
$\mu=\frac{1}{\sqrt[4]{108}}$, suggesting a conjectural upper bound as in Challenge 3 below. 
But possibly there lurk yet flatter $K_3N$'s capable of defeating any $\mu>0$.

Challenge 3: Find a $K_3N$ with $p>\lfloor\sqrt[3]{n}-\sqrt[4]{\frac{n}{108}}-\frac{5}{18}\sqrt[6]{\frac{n}{27}}
+\frac{5}{72}\sqrt[12]{\frac{n}{216}}\rfloor$, or prove that none exists.

\begin{theorem}\label{Th4_7}
For any $K_3N$, 
\begin{equation*}A<\sqrt{3}\sqrt[6]{n}-\frac{1}{2\sqrt[4]{12}}\sqrt[12]{n}\end{equation*}
\end{theorem}
\begin{proof}From Theorem~\ref{Th4_2}, $A<\sqrt{3p}-\frac{1}{2}\sqrt[4]{\frac{p}{12}},$
and from Theorem~\ref{Th4_6} we have $p<\sqrt[3]{n}-\frac{4\sqrt{3}}{9}\sqrt[6]{n}+1<\sqrt[3]{n},$ 
giving the result at once.
\end{proof}
Clearly a slightly tighter bound could easily be found.
For $n^*$ (see \ref{D6}) this gives $A=39<\sqrt{3}\sqrt[6]{n}-\frac{1}{2\sqrt[4]{12}}\sqrt[12]{n}\bumpeq 40.81.$

If $B^\dagger(n)$ is to be a bound for all $K_3N$'s for $B$ 
given $n$ for which $\inf(B^\dagger(n)-B)=0$, 
then $B^\dagger(n)$ as in Theorem~\ref{Th4_8} (a) below suffices.
But if we are content with $\inf(\frac{B^\dagger(n)}{B})=1$ 
then we could much more easily establish $B^\dagger(n)=\sqrt{2}\sqrt[4]{n}$. 
However, we outline a proof of 

\begin{theorem}\label{Th4_8} 
(a) For any $K_3N$, $B<\sqrt{2}\sqrt[4]{n}-\Bigl(\sqrt{3}-\dfrac{1}{2}\Bigr)\sqrt[8]{\dfrac{n}{4}}
+\dfrac{39+\sqrt{3}}{24}$

(b) For any $C_3N$ except 6601, $B<\sqrt{2}\sqrt[4]{n}-\Bigl(\sqrt{7}-\dfrac{1}{2}\Bigr)\sqrt[8]{\dfrac{n}{4}}
+\dfrac{175-3\sqrt{7}}{56}$
\end{theorem}
\begin{proof} We define $B_\mu(n):=\sqrt{2}\sqrt[4]{n}-\Bigl(\sqrt{\mu}-\dfrac{1}{2}\Bigr)\sqrt[8]{\dfrac{n}{4}}
+\dfrac{3\mu^2-\mu\sqrt{\mu}+4\mu+4\sqrt{\mu}}{8\mu}$, and then 
(a) states $B<B_3(n)$ and (b) states $B<B_7(n)$. Also for any $K_£N$, $n\geq 561$, and so
$B_7(n)<B_3(n)$. Further, if we put $u=\sqrt[8]{\dfrac{n}{4}}$, we have
\begin{equation}\label{D23}
 B_\mu(n)=B^*_\mu(u):=2u^2-\Bigl(\sqrt{\mu}-\dfrac{1}{2}\Bigr)u+\dfrac{3\mu+4}{8}-\dfrac{\mu-4}{8\sqrt{\mu}}
\end{equation}
We note that for $\mu\leq7$, $B_\mu$ and $B_\mu^*$ are increasing functions for $n\geq 561$ 
(actually, for $\mu\leq 62.7$).

We consider three cases: (i) $G\geq 2$, (ii) $G=1,F\geq2$ and (iii) $G=F=1$.

(i) For $G\geq 2$, from Theorem~\ref{Th2_2} and (\ref{B19}), $B<p$. Also from (\ref{D22}) 
$H^2>\frac{p}{3}$, so $p\,'(H^2+F)>p\,'(\frac{p}{3}+1)=\frac{p^2+2p-3}{3}>\frac{p^2}{3}$ 
for $p\geq 3$. Hence, with (\ref{B6}) and since $C>B$, we have 
$n>n'=ABCH(H^2+F)>B^2p'(H^2+F)>\frac{B^2p^2}{3}>\frac{B^4}{3}$, so $B<\sqrt[4]{3n}$.
This immediately gives $B<\sqrt{2}\sqrt[4]{n}$, but the tighter bound $B_7(n)$ 
then requires $\sqrt[4]{3n}<B_7(n)$, which holds for $n>5.625 \times 10^9$; 
a computer check verified Theorem~\ref{D8} for $n<5.625 \times 10^9$.

(ii) For $G=1$ from (\ref{D11}) $F$ is odd, so $F\geq 3$, and $H=AF-1$, so $H^2=Fp\,'-H$.
Also $\eta\geq 1$ and from (\ref{B20}) $\frac{p\,'}{\Delta}=\frac{AH}{AG}=H\leq  p\,'$.
Hence, and from (\ref{B15}), we get
\begin{equation*}
\begin{split}
n& =p\{H^2(4p^2+\eta p)+H(5p+\eta-1)+1\}\geq p\{H^2(4p^2+p)+5Hp+1\}\\
& =p\{(Fp\,'-H)(4p^2+p)+5pH+1\}=p\{Fpp\,'(4p+1)-4pp\,'H+1\}\\
& \geq p\{3pp\,'(4p+1)-4pp\,'^2+1\}=8p^4-p^3-7p^2+p
\end{split}
\end{equation*}

So if we put $w:=w(p):=\sqrt[8]{\frac{1}{4}(8p^4-p^3-7p^2+p)}$ we have 
$w \leq u=\sqrt[8]{\frac{n}{4}}$, and with $\mu =7$ in (\ref{D23}) we get
$B^*_7(w)\leq B^*_7(u)=B_7(n)$ for $p \geq 3$. Also from Theorem~\ref{Th4_3} we 
have $B<B(p):=2p-\sqrt{p-\frac{3}{4}}+\frac{\sqrt{3}+1}{2\sqrt{3}}$. 
We now want $B(p)<B^*_7(w)$ for $p\geq 3$, i.e. $\beta(p):=B^*_7(w(p))-B(p)>0$.
The dominant term in $\beta(p)$ is $2(\sqrt[4]{2}-1)p$, so $\beta(p)>0$ 
for sufficiently large $p$, and $\beta(3)\bumpeq 0.596>0$; an outline of a general proof 
is as follows: using the substitution $p=p(v):=v^2+v+1$, we get from $w(p)$ above 
\begin{equation}\label{D24}\begin{split}
w&=w^\dagger(v):=w(p(v))\\
&=\sqrt[8]{\frac{1}{4}\Bigl(8v^8+32v^7+79v^6+125v^5+139v^4+107v^3+54v^2+16v+1\Bigr)}
\end{split}
\end{equation}
and also $B^\dagger_7(v):=B^*_7(w^\dagger(v))$, 
$B_m:=B_m(v):=2p(v)-\sqrt{p(v)-\frac{3}{4}}+\frac{1}{2}=2v^2+v+2$, and then 
\begin{equation}\label{D25}\begin{split}
 \beta^\dagger(v):=&\beta(p(v))=B^\dagger_7(v)-B_m(v)-\frac{1}{2\sqrt{3}}\\
=&2w^2-B_m-(\sqrt{7}-\frac{1}{2})w+\frac{175-3\sqrt{7}}{56}-\frac{1}{2\sqrt{3}}
\end{split}
\end{equation}
\begin{Lemma4A}
If $x>0$, $y>0$ and $x^4>y^4$, then $x>y$ and 
\begin{equation*}
x-y=\frac{x^4-y^4}{x^3+x^2y+xy^2+y^3}>\frac{x^4-y^4}{4x^3}
\end{equation*}.
\end{Lemma4A}
Applying this Lemma in (\ref{D25}) to $2w^2-B_m$, with (\ref{D24}) we get 
\begin{equation*}\begin{split}
2w^2-B_m>&\frac{16w^8-B_m^4}{4\cdot 8w^6}\\
=&\frac{16v^8+96v^7+228v^6+396v^5+411v^4+324v^3+128v^2+32v-12}
{4(32v^8+128v^7+316v^6+500v^5+556v^4+428v^3+216v^2+64v+4)^\frac{3}{4}}
\end{split}
\end{equation*}
Now for $v\geq 10$ we have $50v^8>16(w^\dagger(v))^8$, 
so 
\begin{equation*}
2w^2-B_m>\frac{4v^8+24v^7+57v^6}{50^\frac{3}{4}v^6}>\frac{1}{5}\Bigl(v^2+6v+\frac{57}{4}\Bigr), 
\end{equation*}
and hence from (\ref{D25})
\begin{equation*}\begin{split}
 \beta^\dagger(v)>\frac{1}{5}\Bigl(v^2+6v+\frac{57}{4}\Bigr)-\Bigl(\sqrt{7}-\frac{1}{2}\Bigr)w^\dagger(v)>
\frac{4v^2+24v+57}{20}-\Bigl(\sqrt{7}-\frac{1}{2}\Bigr)\sqrt[8]{\frac{50}{16}}v\\
>\frac{4v^2+24v+57}{20}-\frac{5}{2}v=\frac{4v^2-26v+57}{20}>0\quad\text{ for } v\geq 10
\end{split}
\end{equation*}

For $v=10$, $p=111$, and a computer check verified Theorem~\ref{Th4_8} for all $K_3N$'s with $3\leq p\leq 109$.

(iii) Let $n_o$ be a $K_3N$ with associated $K$-variables $H_o, A_o, p_o$, etc, with $F_o=G_o=1$.
Then from (\ref{B18}) $A_o=H_o+1$, so $p_o=H_o^2+H_o+1$, from (\ref{B20})
$\Delta_o=A_oG_o=H_o+1$, from (\ref{B14a}) 
$B_o=2p_o+\frac{\eta_o-\theta_o}{2}$, and $\Delta_o<P_o$, so the conditions needed for the proof of Theorem~\ref{Th3_1} 
and case (ii) of Theorem~\ref{Th3_5} are met, and in a similar manner we define real variables $n,p,A,B$ etc 
in terms of independent real variables $H$ and $\eta$. 
Then, keeping $H=H_o$, constant, we get from the above with (\ref{B13a}), (\ref{B14a}) and (\ref{B15})
\begin{equation}\label{D26}
 B=B(H,\eta):=2H^2+2H+2+\frac{1}{2}(\eta-\sqrt{4\eta H^2+4\eta'(H+1)+\eta^2})
\end{equation}
and
\begin{equation}\label{D27}\begin{split}
 n=n(H,\eta):=4H^8+&12H^7+(24+\eta)H^6+(33+2\eta)H^5+(34+3\eta)H^4\\
+&(26+3\eta)H^3+(14+2\eta)H^2+(5+\eta)H+1
\end{split}
\end{equation}
Then as for Theorems~\ref{Th3_1} and \ref{Th3_5} we have that as $\eta$ increases, so $B$ decreases and $n$ increases, 
so if $1\leq \eta_a\leq \eta_o$, then $B(H_o,\eta_o)\leq B(H_o,\eta_a)$ and $n(H_o,\eta_o)\geq n(H_o,\eta_a)$. 
So for the best possible bound we want to choose $\eta_a$ to be the smallest possible $\eta_o$ 
(we already know from the proof of Theorem~\ref{D3} that $\eta_o=1=F_o=G_o$ is impossible for $K_3N$'s). 
Dropping the zero suffixes for our $K_3N$, from (\ref{B12}) we have $s^2+\eta s=\eta(H^2+H+1)-(H+1)=\eta H^2+\eta' H+\eta'$ 
and with (\ref{B13b}) we get
\begin{equation}\label{D28}
 \phi:=2\eta H+\eta-1\quad \text{ and } \quad \theta=2s+\eta\quad\text{ and then}
\end{equation}
\begin{equation}\label{D29}
 \phi^2-\eta\theta^2=-(\eta^3+3\eta^2-2\eta-1)
\end{equation}
Compare this with (\ref{D14}) and (\ref{D15}) and the accompanying discussion of Pellian solutions, 
which we now apply to (\ref{D28}) and (\ref{D29}). We define \emph{acceptable} 
solutions to (\ref{D29}) in the same way as for (\ref{D15}), and for a solution to be \emph{admissible} 
we require $\theta-\eta$ even, and $\phi\equiv\eta -1 \pmod {4\eta}$ since in (\ref{D28}) we have $H$ even. 
Henceforth replacing $F$ with $\eta$ in (\ref{D16}) with $x^2-\eta y^2=1$, if $\eta$ is not a perfect square 
and (\ref{D29}) has a solution, then (\ref{D16}) gives an infinity of further solutions.
We find that there are no solutions for $\eta=1,2\text{ or }5$; for $\eta=4$ there is the unique solution 
$\phi=51, \theta=26$ which leads to the $K_3N$ $43\cdot 451\cdot 607=11771551$ with $B=75$; while for each 
$\eta\in\{3,6,7\}$ there are four fundamental solutions from which all other solutions can be derived 
via (\ref{D16}), but the only admissible ones are $\phi=14, \theta=9$ for $\eta=3$ and 
$\phi=90, \theta=35$ for $\eta=7$, for which $\phi^2-7\theta^2=-475$ reduces to $X^2-7Y^2=-19$ 
via $\phi=5X, \theta=5Y$.

With reference to (\ref{D16}), amended as above, we need to show that for $\eta\leq 7$ 
any solution $(\phi_i,\theta_i)$ is admissible if{}f its fundamental solution $(\phi_1,\theta_1)$ 
is admissible.
From (\ref{D16}) $\theta_{i+1}\equiv \theta_i \pmod 2$, so $(\theta_{i+1}-\eta)$ is even if{}f 
$(\theta_i-\eta)$ is even.
Also the inverse transformation for (\ref{D16}) has $\phi_i=(2\eta y^2+1)\phi_{i+1}-2\eta xy\theta_{i+1}$, 
and we have $\eta y^2=(x-1)(x+1)$, so if $xy$ is odd then $8\mid\eta$, whence $xy$ is even for $\eta\leq 7$. 
Working now in $\mathbb{Z}_{4\eta}$ with $\eta\leq 7$, if $y$ is even then $\phi_{i+1}=\phi_i$, 
and if $\eta$ is even then $x$ is odd and $y$ is even. 
But if odd $\eta=2\nu +1$ and $y$ is odd, suppose $\phi_i=\eta-1$; then 
$\phi_{i+1}=(2\eta y^2+1)\phi_i=(2\eta+1)(\eta-1)=2\nu(2\eta+1)=4\nu\eta+\eta-1=\eta-1$;
conversely by the inverse transformation $\phi_i=(2\eta y^2+1)\phi_{i+1}$, so $\phi_{i+1}=\eta-1$ 
if{}f $\phi_i=\eta-1$, whence by induction $\phi_i=\eta-1$ if{}f $\phi_1=\eta-1$. 
Thus as required $(\phi_i,\theta_i)$ is admissible if{}f $(\phi_1,\theta_1)$ is admissible. 
Also for acceptibility we require $E\geq2$: it is easily shown that $E=F=G=1$ gives rise to 
$\eta=H^4+2H^3+H^2+H+1$ (cf~(\ref{D18}) and Theorem~\ref{Th4_4}), so $\eta=39$ (with $H=2$) 
is the least $\eta$ with an admissible solution with $E=1$; thus $E\geq2$ for $\eta\leq7$.

With $\eta=3$, the solution $\phi=14, \theta=9$ gives the $C_3N\quad 7\cdot 23\cdot 41=6601$, 
but this is the only $C_3N$, since we find in $\mathbb{Z}_7$ that the cycles given by (\ref{D16}) 
are of period 4, and in $\mathbb{Z}$ with obvious notation 
$\gcd(A_{4i+2},B_{4i+2},C_{4i+2})=7$, not acceptable, 
and $p_{4i+1}\equiv p_{4i+3}\equiv q_{4i}\equiv r_{4i}\equiv 0 \pmod 7$.
This is sufficient for our proof, but for $\eta=7$ the solution $\phi=90, \theta=35$ gives the 
$C_3N\quad 43\cdot 433\cdot 643=11972017$ with $B=72$, and there seems to be the \emph{possibility} 
of further $C_3N$'s in the sequence of $K_3N$'s generated by (\ref{D16}).

It remains to show that for $\eta=3\text{ or }7$ and any even $H\geq 2$, $B(H,\eta)<B_\eta(n(H,\eta))$. 
If we write (\ref{D26}) as $B=2H^2+2H+\frac{4+\eta}{2}-\sqrt{\eta}H(1+x(H,\eta))^\frac{1}{2}$ 
and (\ref{D27}) as $\frac{n}{4}=H^8(1+y(H,\eta))$, then for large $H$ with $x:=x(H,\eta)$ and 
$y:=y(H,\eta)$ we have $x=O(\frac{1}{H}), y=O(\frac{1}{H})$ and $B_\eta(n)=2H^2(1+y)^\frac{1}{4}-H(\sqrt{\eta}-\frac{1}{2})(1+y)^\frac{1}{8}
+\frac{3\eta+4}{8}-\frac{\eta-4}{8\sqrt{\eta}}$. 
Using the binomial series and omitting the complicated details we get 
\begin{equation*}
 B_\eta(n)-B=\Bigl(\frac{3\eta\sqrt{\eta}}{32}-\frac{\eta}{64}+\frac{15\sqrt{\eta}}{128}+\frac{117}{256}
-\frac{1}{4\sqrt{\eta}}-\frac{1}{8\eta\sqrt{\eta}}\Bigr)\frac{1}{H}+O\Bigl(\frac{1}{H^2}\Bigr)>0
\end{equation*}
for sufficiently large H.
To prove the result rigorously for all even $H=H_o$, we can truncate the various binomial series 
and approximate $x$ and $y$ to form functions $b_\eta(n(\eta,H))=b_\eta(n)$ and $b(\eta,H)=b$ 
such that $B_\eta(n)>b_\eta(n)$ and $b>B$, with $b_\eta(n)$ and $b$ containing a relatively 
small number of terms all of which are retained, whence for each $\eta$ we can determine a precise 
$H^*(\eta)$ such that $b_\eta(n)-b>0$ for $H\geq H^*(\eta)$. 
My method of truncation and approximation was as arithmetically economical as I could make it, 
subject to retaining exactly the above term in $\frac{1}{H}$, and after heavy detail arrived at 
$H^*(3)=240$ and $H^*(7)=66$; Gordon Davies (see \S\ref{Sec5a}) did the computer verifications 
for $2\leq H\leq H^*(\eta)$.

Thus for any $K_3N$ with $F_o=G_o=1$, we have $\eta_o\geq3$ and for any $C_3N$ except 6601 
we have $\eta_o\geq 7$, so with $\eta_a=3 \text{ or } 7$ as appropriate we have $\eta_o\geq\eta_a$ and 
$B_o=B(H_o,\eta_o)\leq B(H_o,\eta_a)<B_{\eta_a}(n(H_o,\eta_a))\leq B_{\eta_a}(n(H_o,\eta_o))=B_{\eta_a}(n_o)$ 
as required.
\end{proof}

We can construct a $K_3$-family $\{n_4(t)\}$ such that, for any fixed $\mu\geq1$ and large~$t$, 
$B(t)\sim B_\mu(n_4(t))\sim B_{\eta(t)}(n_4(t))\sim\sqrt{2}\sqrt[4]{n_4(t)}$ as follows: 
$F=G=1$ and then $H=t(t+1),A=\eta=t^2+t+1, s=tA$, giving $p_4(t)=t^4+2t^3+2t^2+t+1$, 
$B(t)=2t^4+3t^3+3t^2+t+2$, $C(t)=2t^4+5t^3+6t^2+4t+3$, $q_4(t)=2t^6+5t^5+6t^4+4t^3+3t^2+2t+1$,
$r_4(t)=2t^6+7t^5+11t^4+10t^3+7t^2+3t+1$ and $n_4(t)=p_4(t)\cdot q_4(t)\cdot r_4(t)$. 
Then we have $B\sim 2t^4$ and $n\sim4t^{16}$, so $B\sim \sqrt{2}\sqrt[4]{n}\sim B_\mu(n)\sim B_\eta(n)$. 
Then $n_4(1)=6601$ and $n_4(2)=11972017$ as above, but we found no more $C_3N$'s up to $t=31$. 
Obviously (\ref{D29}) is satisfied identically by the parametric forms for $H,\eta$ and $s$ 
of $\{n_4(t)\}$, since $F=G=1$.

For $\eta=7$ and $(\phi_1,\theta_1)=(90,35)$, (\ref{D16}) gives $(\phi_2,\theta_2)=(23190,8765)$ 
which yields the $K_3N\quad n=2743993\cdot 9080853193\cdot 9095368033$ with $B=5483607$, $H=1656$ 
and $B_7(n)=5483607.001$ (calculator accuracy).

We note that for the sequence of $K_3N$'s associated with Theorem~\ref{Th4_3} and (\ref{D16}) for given $F\geq 3$, 
$B\sim \sqrt{2}\sqrt[4]{\frac{n}{F}}$.
I confidently conjecture from the above proof of Theorem~\ref{Th4_8} and from numerical evidence that $B<B_\eta(n)$ 
for every $K_3N$, but have not attempted a general proof.

Challenge 4: Find a $C_3N$ for which $n>6601$ and $B_7(n)-B<0.1$.

\medskip
To conclude \textsection\ref{Sec4} we shall show that the $K_3$- families deriving from (\ref{D18}) and (\ref{D19}),
with $E=2$ and $F=G=1$, not only give equality for the upper bound for $C$ given $p$, as described
in the discussion following Theorem \ref{Th4_4}, but in the same way give equality
for upper bounds for $C, BC, ABC$ and $n$ given $H$, and for $C$ and $ABC$ given $n$.
Using the notation of (\ref{D18}) and (\ref{D19}):

\begin{theorem}\label{Th4_9}For any $K_3N$, (a) $C\leq C(H)$, (b) $BC\leq B(H) \cdot C(H)$,
(c) $ABC \leq A(H) \cdot B(H) \cdot C(H)$ and (d) $n\leq N(H)$, with equality if{}f $E=2$ and $F=G=1$.
If there is equality then $H\not\equiv  2 \pmod 6$, and  if also \,$n$ is a $C_3N$, then $H \equiv 0 \pmod 6$.
\end{theorem}

\begin{proof}From (\ref{B18}), $A=\frac{H+G}{F}$, and then via (\ref{B19}), (\ref{B16}) and (\ref{B6})
we can express $B,C,n$ ( and obviously also $p,q,r$) in terms of $E,F,G,H$.
We now regard all other variables as functions of continuous real
variables $E,F,G,H$, subject to all the relations so far established for $K_3N$'s.
Let $Z$ be any of, or any product from, $A,B,C$ and $n'$: regarding $H$ as fixed,
we write $\phi_Z(E,F,G):=Z$ evaluated at $(E,F,G)$. From (\ref{B6}) and (\ref{B18}) we have
$n'=ABCH^3+BCH(G+H)$, and then we readily see that $F$ occurs only in the denominator of any $Z$, and thence
\begin{equation}\label{D30}\phi_Z(E,F,G)< \phi_Z(E,1,G)\quad\text{unless}\quad F=1
\end{equation}
\begin{Lemma4B}If $f(x):=ax+\dfrac{b}{x}$ with $a>0, b>0$, then for $0<x_1<\sqrt{\dfrac{b}{a}}$
we have $f(x_1)>f(x)$ if{}f $x_1<x<\dfrac{b}{ax_1}$.
\end{Lemma4B}

This follows at once from 
\begin{equation*}f(x_1)-f(x)=a(x_1-x)+b\Bigl(\frac{1}{x_1}-\frac{1}{x}\Bigr)=(x-x_1)\Bigl(\frac{b}{xx_1}-a\Bigr)
\end{equation*}

From (\ref{B16}, \ref{B18}, \ref{B19}) we have 
\begin{subequations}\label{D31}
\begin{multline}\label{D31a}C=\frac{ABH+A+B}{E}=\frac{1}{E}\Bigl\{\Bigl(\frac{H+G}{F}\Bigr)\Bigl(\frac{H^2+GH}{FG}
+\frac{E+1}{G}\Bigr)H\\+\frac{H+G}{F}+\Bigl(\frac{H^2+GH}{FG}+\frac{E+1}{G}\Bigr)\Bigr\}\ ,\qquad\text{i.e.}
\end{multline}
\begin{multline}\label{D31b}C=\phi_C(E,F,G)=\frac{1}{EF^2}\Bigl\{(H^2+F)G+\frac{H^2(H^2+2F)}{G}\Bigr\}
\\+\frac{2H^3}{EF^2}+\frac{H}{F}+\frac{3H}{EF}+\frac{1}{G}+\frac{1}{EG}+\frac{H^2}{FG}
\end{multline}
\end{subequations}
Obviously since $E\geq 2$, 
\begin{equation}\label{D32}\phi_C(E,F,G)<\phi_C(2,F,G)\quad\text{unless}\quad E=2\end{equation}
Similarly, the last three terms of (\ref{D31b}) increase as $G$ decreases.
Also, applying Lemma~4B with $x=G, f(G):=(H^2+F)G+\frac{H^2(H^2+2F)}{G}$ and $x_1=1$ since $G\geq 1$,
we have $f(1)>f(G)$ if{}f $1<G<\frac{H^2(H^2+2F)}{H^2+F}$. But from Theorem \ref{Th2_3}, $G<2H$,
and since $H\geq 2$ we have $1\leq G<2H\leq H^2<H^2(\frac{H^2+2F}{H^2+F})$, so $f(1)>f(G)$ unless $G=1$.
Hence from (\ref{D31b}),
\begin{equation}\label{D33}\phi_C(E,F,G)<\phi_C(E,F,1)\quad\text{unless}\quad G=1\end{equation}
Applying (\ref{D30}, \ref{D32}, \ref{D33}) in turn,
\begin{equation}\label{D34}C=\phi_C(E,F,G)<\phi_C(2,1,1)\quad\text{unless}\quad E=2\text{ and }F=G=1\end{equation}
Combining this with the discussion of (\ref{D18}) and (\ref{D19}), we have Theorem \ref{Th4_9}(a).

Also $B=\frac{H}{F}+\frac{H^2}{FG}+\frac{E+1}{G}$ (as in (\ref{D31a})), so $\phi_B(E,F,G)<\phi_B(E,F,1)$ unless $G=1$,
and hence with (\ref{D33})
\begin{equation}\label{D35}BC=\phi_{BC}(E,F,G)<\phi_{BC}(E,F,1)\quad\text{unless}\quad G=1\end{equation}

Now given $F, G$ and $H$, also determined are $A=\frac{H+G}{F}$ and $p=AH+1$.
Hence from (\ref{B19}) and (\ref{B16}) we have
\begin{multline*}\phi_{BC}(E,F,G)=\Bigl(\frac{p+E}{G}\Bigr)\Bigl(\frac{Bp+A}{E}\Bigr)=\frac{p+E}{G}\Bigl(\frac{p\,^2+pE+AG}{EG}\Bigr)
\\=\frac{p}{G^2}\Bigl\{E+\frac{p\,^2+AG}{E}\Bigr\}+\frac{2p\,^2+AG}{G^2}\end{multline*}
Applying Lemma~4B with $x=E$, $f(E)=E+\frac{p\,^2+AG}{E}$ and $x_1=2$ since $E\geq 2$,we have $f(2)>f(E)$
if{}f $2<E<\frac{p\,^2+AG}{2}$, which is true unless $E=2$, since $p\geq 3$ and by Theorem \ref{Th2_2}, $E\leq p-1$. Thus
\begin{equation}\label{D36}\phi_{BC}(E,F,G)<\phi_{BC}(2,F,G)\quad\text{unless}\quad E=2\end{equation}
Hence from (\ref{D30}, \ref{D35}, \ref{D36}), 
\begin{equation}\label{D37}BC=\phi_{BC}(E,F,G)<\phi_{BC}(2,1,1)\quad\text{unless}\quad E=2\text{ and } F=G=1,
\end{equation}
and Theorem \ref{Th4_9}(b) follows as for \ref{Th4_9}(a).

Also $A$ is independent of $E$, so from (\ref{D36}) 
\begin{equation}\label{D38}\phi_{ABC}(E,F,G)<\phi_{ABC}(2,F,G)\quad\text{unless}\quad E=2\end{equation}
\vskip-16pt
\begin{multline*}\text{Further, }AB=\phi_{AB}(E,F,G)=\frac{H+G}{F}\Bigl(\frac{H^2+GH+F(E+1)}{FG}\Bigr)
\\=\frac{H}{F^2}\Bigl(G+\frac{H^2+F(E+1)}{G}\Bigr)+\frac{2H^2+F(E+1)}{F^2},\end{multline*}
so applying Lemma~4B with $x=G$, $f(G)=G+\frac{H^2+F(E+1)}{G}$ and $x_1=1$ since $G\geq 1$,
we have $f(1)>f(G)$ if{}f $1<G<H^2+F(E+1)$, which is true unless $G=1$, since
$H\geq 2$ and from Theorem \ref{Th2_3}, $G<2H$. Hence
\begin{equation}\label{D39}\phi_{AB}(E,F,G)<\phi_{AB}(E,F,1)\quad\text{unless}\quad G=1.\end{equation}

Then from (\ref{D33}) and (\ref{D39}),
\begin{equation}\label{D40}\phi_{ABC}(E,F,G)<\phi_{ABC}(E,F,1)
\quad\text{unless}\quad G=1, \end{equation}
and from (\ref{D30}, \ref{D38}, \ref{D40})
\begin{equation}\label{D41}ABC=\phi_{ABC}(E,F,G)<\phi_{ABC}(2,1,1)\quad\text{unless}\quad E=2\text{ and } F=G=1
\end{equation}
and Theorem \ref{Th4_9}(c) follows as for \ref{Th4_9}(a).

Also from (\ref{D38}) and (\ref{D40}), $\phi_{ABC}(E,F,G)<\phi_{ABC}(2,F,1)$ unless $E=2, G=1$
and from (\ref{B6}) $n'=ABCH(H^2+F)$, so
\begin{equation*}\phi_{n'}(E,F,G)=\phi_{ABC}(E,F,G)H(H^2+F)<\phi_{ABC}(2,F,1)H(H^2+F)=\phi_{n'}(2,F,1)\end{equation*}
 unless $E=2, G=1$;
and from (\ref{D30}) $\phi_{n'}(2,F,1)<\phi_{n'}(2,1,1)$ unless $F=1$.
Thus $n'=\phi_{n'}(E,F,G)<\phi_{n'}(2,1,1)$ unless $E=2$ and $F=G=1$; Theorem~\ref{Th4_9} follows as for~\ref{Th4_9}(a).
\end{proof}

We now express our results for upper bounds for $C$ and $ABC$ given $n$ in terms of the functions of (\ref{D18})
and the inverse function $N^{-1}$ of (\ref{D19}); using the method of Theorem \ref{Th3_5},
we could also express our results as series of descending powers of $\sqrt[10]{n}$, but we simply
indicate the leading terms:
\begin{theorem}\label{Th4_10} For any $K_3N$, (a) $C\leq C(N^{-1}(n))$ and 
\newline (b) $ABC\leq A(N^{-1}(n)) \cdot B(N^{-1}(n)) \cdot C(N^{-1}(n))$
with equality as in Theorem~\ref{Th4_9}.
For large $n$, $C\leq 2^{-\frac{3}{5}}n^{\frac{2}{5}}(1+o(1))$, $ABC\leq 2^{-\frac{3}{10}}n^{\frac{7}{10}}(1+o(1))$.
\end{theorem}
\begin{proof}
Suppose that all variables are real as in the proof of Theorem~\ref{Th4_9}, and that $(E,F,G)=(2,1,1)$;
then $n=N(H), H=N^{-1}(n)$, so $A=A(N^{-1}(n))$, etc. 
Also the first part of the proof of Theorem~\ref{Th3_6} applies,
with $i=d=3$ and $\lambda_d=E=2$, giving $r=\frac{\sqrt{8n+1}+1}{4}$. But $C=\frac{r'}{H}$, so we have
the functional relationship
\begin{equation}\label{D42}C(N^{-1}(n))=\frac{\sqrt{8n+1}-3}{4N^{-1}(n)}\end{equation}
Similarly from (\ref{B6}) we get
\begin{equation}\label{D43}A(N^{-1}(n)) \cdot B(N^{-1}(n)) \cdot C(N^{-1}(n))=\frac{n'}{N^{-1}(n)\{(N^{-1}(n))^2+1\}}
\end{equation}

Now suppose that $n$ is any $K_3N$ with its standard $A$ to $H$ integer set. Then (a) $C=\frac{r'}{H}$,
from Theorem~\ref{Th3_7} $r'\leq\frac{\sqrt{8n+1}-3}{4}$ with equality if{}f $E=2$ and from Theorem~\ref{Th4_9}
$H\geq N^{-1}(n)$ with equality if{}f
$(E,F,G)=(2,1,1)$, and hence
\newline $C\leq \frac{\sqrt{8n+1}-3}{4N^{-1}(n)}=C(N^{-1}(n))$ from~(\ref{D42}), with equality if{}f $(E,F,G)=(2,1,1)$.
Also (b) from~(\ref{B6})
\begin{equation*}ABC=\frac{n'}{H(H^2+F)}\leq \frac{n'}{N^{-1}(n)\{(N^{-1}(n))^2+1\}}=
A(N^{-1}(n)) \cdot B(N^{-1}(n)) \cdot C(N^{-1}(n))\end{equation*}
from~(\ref{D43}), with equality if{}f $(E,F,G)=(2,1,1)$. The first sentence of Theorem~\ref{Th4_10} follows.

Also for large $n$ and equality, from~(\ref{D18}) and~(\ref{D19}) we have $C\sim \frac{1}{2}H^4$,
\hbox{$ABC\sim \frac{1}{2}H^7$}
 and $n\sim \frac{1}{2}H^{10}$, so $H\sim (2n)^{\frac{1}{10}}$ and
$C\sim 2^{-\frac{3}{5}}n^{\frac{2}{5}}$, $ABC\sim 2^{-\frac{3}{10}}n^{\frac{7}{10}}$,
giving the final part of Theorem~\ref{Th4_10}.
\end{proof}

\section{The algorithm and its implementation}\label{Sec5}
\subsection{Acknowledgments}
\label{Sec5a}
This is the most convenient place to acknowledge my immense debt to my two friends who have carried out the
computer implementation of my algorithms: firstly Gordon Davies, like me a retired teacher of mathematics at
Haileybury College, England, who did the computing throughout the development stage,
programming in BASIC V with 32-bit arithmetic, and by August 1999 taking us on his RISC-PC, running RISC OS 3.7 with 16 Mb of RAM, to $C_3$ ($2\times 10^{18}$)~$=42720$
(where $C_3(X)$ is the number of three-prime Carmichael numbers up to $X$); and secondly Matthew Williams, a recent Haileybury student and Cambridge
University computer science graduate, who then got us up to $C_3(10^{24})$, using his 1 GHz Athlon with 900 Mb of RAM
and programming in C++ mostly with 64-bit arithmetic.

\subsection{Notation}
\label{Sec5b}
In this section we shall use the following upper bounds based on results in \S~\ref{Sec3} and \S~\ref{Sec4}. 
Each of these could be replaced with a slightly greater and simpler bound with only marginal loss 
of efficiency.
We seek all $C_3N$'s less than $X$, where for convenience $X$ is not a $C_3N$, so $n<X$.

(a) $p\leq p_M:=\left\lceil\sqrt[3]{X}-\dfrac{4\sqrt{3}}{9}\sqrt[6]{X}\right\rceil\qquad$
(Theorem~\ref{Th4_6})

(b) $A<A_M:=\sqrt{3p\,'}-\frac{1}{2}\sqrt[4]{\dfrac{p\,'}{12}}\qquad$ 
(Theorem~\ref{Th4_2})

(c) $B<B_M:=2p-\sqrt{p-\dfrac{3}{4}}+\dfrac{\sqrt{7}+1}{2\sqrt{7}}\qquad$ 
(Theorem~\ref{Th4_3})

(d) $q\leq Q_1:=P\,'\left(2P-\sqrt{P-\dfrac{3}{4}}+\dfrac{1}{2}\right)+1\qquad$ 
(Theorem~\ref{Th3_1})

(e) $q<Q_2:=\sqrt{\dfrac{X}{p}}-\sqrt{\dfrac{p}{12}}\qquad$ 
(Theorem~\ref{Th4_5})

(f) $Z:=\min(Q_1,Q_2)$

(g) For any $CN$ $q\leq Q_3:=\left\lfloor\sqrt[5]{2X^2}-\sqrt[10]{\dfrac{X^3}{64}}
-\dfrac{1}{10}\sqrt[5]{\dfrac{X}{4}}+\dfrac{17}{20}\sqrt[10]{\dfrac{X}{4}}\right\rfloor$ 
(Theorem~\ref{Th3_5})

\subsection{A brief description of four algorithms and the split range procedure}\label{Sec5c}
In~\cite{Pinch} Pinch describes the two algorithms which he used  to find all $CN$'s up to $10^{15}$.
We briefly describe these now, since we used them, slightly modified to take advantage of $d=3$,
to do selective checks on our results for large $X$.
In all these algorithms, for given $X$ the outermost loop runs through all odd primes up to~$p_M$.

In Pinch's first algorithm, which as modified by me for $d=3$ we call PI, for each $p$,
$E$ runs through the range $E_L(p,X)\leq E\leq p\,'$ (Theorem~\ref{Th2_2}), where $E_L(p,X)$ is
a fairly complicated function, not given here, which I formulated using $d=3$
(so $P=p<q$) to cut out some of the cases which would result in $n>X$ (for \hbox{$p<\sqrt[4]{3X}$}, $E_L(p\,,X)=2)$;
using $E_L(p,X)$ reduced the time for PI by about a third. For each $E$ a range of integer values of $D$ is found,
subject to $1\leq \Delta=DE-p\,^2\leq 2p\,'$ (Theorem~\ref{Th2_3}), and for each $(E,D)$ pair $q$ and $r$ are calculated
from~(\ref{B9}). $E$, and $D$ for each $E$ descend through their ranges, 
and if $r>\dfrac{X}{pq}$ next $E$ is taken;
else $q,r$ and $\lambda_1=\dfrac{qr-1}{p\,'}$ (\S\ref{Sec2a}) must be integers, 
with next $D$ at the failure of any test, and $q$ and $r$ are tested for primality.

For large $X$ the remaining algorithms are all significantly speeded up by the split-range procedure,
which we briefly describe. Suppose that variables $x$ and $y$ are connected by the bilinear relation
$axy+bx+cy+d=0$ with $a>0,\triangledown:=bc-ad>0$, and that $x_1>-\dfrac{c}{a}$, so over the interval
$x_1\leq x\leq x_2$, $y$ decreases as $x$ increases; and also that we wish to find integer pairs $(x,y)$ 
over this interval, and that a trial where we start with $x\in \mathbb{Z}$ (an $x$-trial) costs $k$ times
the cost of a $y$-trial. Then if $\dfrac{dy}{dx}=-k$ at $(\xi,\psi)$ we minimise the cost
by using $x$-trials for $x<\xi$ and $y$-trials for $y<\psi$ (so $x>\xi$). So if $x_1<\xi<x_2$
it pays to split the range at $(\xi,\psi)=\Bigl(\dfrac{\sqrt{\triangledown}-\sqrt{k}c}{\sqrt{k}a},
\dfrac{\sqrt{k\triangledown}-b}{a}\Bigr)$ (as may easily be shown), with $a$ saving,
compared with using only $x$-trials and ignoring the cost of deciding
whether to split, of approximately $\dfrac{ka(x_2-\xi)^2}{ax_2+c}$ $y$-trial costs.

Pinch's second algorithm as modified by me for $d=3$ (PII) for each $p$ runs through all primes $q$ satisfying $p<q\leq Z$.
For each $(p,q)$ pair it uses the Euclidean algorithm to find $H$ and hence $A=\frac{p\,'}{H}$, and if $A>A_M$ it takes next $q$.
Else it finds $L_1:=\dfrac{p\,'q\,'}{H}=\lcm[p\,',q\,'], R=pq,$ and $w$ such that $wR\equiv 1\pmod {L_1},$ 
by the reverse Euclidean algorithm; then, since $n=rR\equiv 1\pmod{L_1},$ we have $r=w+uL_1,$ and also $R\,'=Er'$ 
(\ref{B1}), so eliminating $r$ we seek integer pairs $(u,E)$ such that 
\begin{equation}\label{E1}L_1Eu+w'E-R'=0.\end{equation}
With $u$ ascending and $E$ descending we use the split range procedure, take next $q$ when $r>\dfrac{X}{pq},$
and for each integer pair $(u,E)$ we test $r$ for primality.

Our first successful algorithm (HI, originally devised when seeking Perrin pseudo\-primes, before we knew of
other algorithms) is the same as PII as far as finding $A<A_M$. It then found $r$ in essentially the same way as our main algorithm
HII, described next.

HII was motivated by the realisation that as $p$ becomes larger in HI many more pairs $(p,q)$ will result in $H$
small enough to give $A>A_M$; and that by first analysing $p\,'=AH$ such pairs need never be considered.
Since $H$ is even, $A$ divides $\dfrac{p\,'}{2}$ and, taking the most favourable case as an example, if $\dfrac{p\,'}{2}$ is prime
(and so is a Sophie Germain prime) then $A=1, H=p\,'$ is the only possibility and all possible values of $q$ 
will belong to the arithmetic progression (AP) $q\equiv p\pmod{p\,'},$ with $q\leq Z$. So for each $p$
we factorise $\frac{p\,'}{2}$ to find all possible pairs $(A,H)$ with $A<A_M$, and we organise a set of AP's which
will contain without repetition the resulting $q$-values, which we then test for primality. For each prime $q$ we find $B=\dfrac{q\,'}{H}$;
we find integer pairs $(C,F)$ from~(\ref{B5a}), which is bilinear in $C$ and $F$, determining the range of $F$
as described below in \S~\ref{Sec5d}, and test $r=CH+1$ for primality.

\subsection{The implementation of the HII algorithm in more detail}\label{Sec5d}
The basic idea of the HII algorithm is as stated above, and we now describe more fully a method of programming it,
drawing attention to certain worthwhile economies (our own method was slightly more complicated,
as we shall explain briefly in the next subsection). We structure our description in terms of several loops,
beginning with the outermost as Loop 1. To find all $C_3N$'s with $n<X$, we begin by preparing a bitmap prime database
up to at least $Q_3$ and we calculate $p_M$.

{\bf Loop 1:} for each $p$ satisfying $3\leq p\leq p_M$, we calculate $A_M, B_M$ and $Q_2$, and form an array
$\{A(j)\}$ of possible $A$ values. To do this we have $A\,|\,\dfrac{p\,'}{2}$, so suppose that the prime factorisation is
$\dfrac{p\,'}{2}=\prod_{k=1}^\mu\rho_k{}^{\eta_k},$ where $\eta_k\geq 1$.
We have $A(1)=1$, and for $j\geq 2$ we can obtain $A(j)$ by a process involving successive trial division of $\dfrac{p\,'}{2}$
by ascending primes, for each new $\rho_k$ multiplying all $A(j)$ already found by $\rho_k$ to form more possible $A$ values
(but taking care to avoid duplication when $\eta_k\geq 2$), testing for $A<A_M$ before adjoining $A$ to the array,
and storing the successive prime factors
$\{\rho\}:=\{\rho_1,\rho_2,\dotsc,\rho_\delta\}$ of $\{A(j)\}$ as
they arise (so $\delta\leq\mu$).

{\bf Loop 2:} for each $A(j)=A$, we now develop a set of AP's which must contain $q$ for any $C_3N$ associated
with $(p,A)$, with each AP having common difference $p\,'$, and we also obtain the corresponding set of AP's 
with common difference $A$ which contains the associated $B$ values.
We denote the $\lambda^{th}$ term of the $i^{th}$ AP for $q$ by $q_{\lambda}(i)$, 
so $q_{\lambda}(i)=q_1(i)+\lambda' p\,'$, 
and similarly for $B_\lambda(i)$. 
We have $H=\frac{p\,'}{A}$, and for use in Loop 3 we define $F_o=\left\lfloor\frac{H}{A}\right\rfloor$ 
and $\nu:=(F_o+1)A-H$. 
Clearly $B=\frac{q\,'}{H}<\frac{Q_2}{H}$ and $B<B_M$; also in loop 3 we shall show that $B<\frac{2p}{\nu}$, 
so an upper bound for $B$ is $\beta:=\min(B_M,\frac{Q_2}{H},\frac{2p}{\nu})$. 
Next we use the facts that $q$ is prime and $B$ is coprime to $A$ to restrict the number of cases 
to be considered. 
For $A>1$ we use a sieving method with those $\rho_k\in \{\rho\}$ which divide $A$ to find 
$\Phi(A):=\{t:1\leq t\leq A'$ and $\gcd[A,t]=1\}$ and we define $\Phi(1):=\{1\}$. 
Then for each $t$ we form $B_1=A+t$ and $q_1=B_1H+1=AH+tH+1$, and for $A\geq 3$, if 
$H\not\equiv0\pmod {\rho_k}\ \exists\ t\in\Phi(A)$ such that $t\equiv-\frac{1}{H}\pmod{\rho_k}$ 
and then $q_1\equiv0\pmod{\rho_k}$, so $q_\lambda=q_1+\lambda'p\,'=q_1+\lambda'AH\equiv0\pmod{\rho_k}$ 
and thus the AP $\{q_\lambda\}$ is entirely composite. 
Therefore if $\rho_k\mid  q_1$ we do not adjoin $B_1$ or $q_1$ to the arrays $\{B_1(i)\}$ or 
$\{q_1(i)\}$ (Gordon found that up to $X=10^{18}$ this eliminated from consideration about 19\% 
of the potential AP's). 
So we form the arrays $\{B_1(i)\}$ and $\{q_1(i)\}$ of first terms and we use the iterations 
$B_\lambda(i)=B_{\lambda-1}(i)+A$ and $q_\lambda(i)=q_{\lambda-1}(i)+p\,'$ to form arrays 
$\{B_\lambda(i)\}$ and $\{q_\lambda(i)\}$, having first examined each $B_{\lambda-1}(i)$ 
for associated $C_3N$'s as described in Loop 3 below. 
Having used a sieving method to find $\Phi(A)$, we get $B_\lambda(i)$ and $q_\lambda(i)$ 
increasing steadily throughout this process, and $B_\lambda(i)>\beta$ triggers next $A(j)$.

{\bf Loop 3:} We write $B_{\lambda-1}(i)=B$ and $q_{\lambda-1}(i)=q$, and if $q$ is composite 
(check against bitmap prime database), we take next $B$ and $q$. 
Writing $K:=AB, U:=KH+A+B=Bp+A$ and $V:=(A+B)H+1=p+q\,'$, (\ref{B5a}) may be written 
\begin{equation}\label{E2}
 KCF-VC-U=0
\end{equation}
and we use this bilinear relation in $C$ and $F$ to find integer pairs $(C,F)$ 
and hence possible $r=CH+1$.
Theorems~\ref{Th2_1} and \ref{Th4_4} suggest that we consider the $F$-range, and we examine 
certain related economies including a procedure for splitting the range.

Put $Y:=\frac{X}{pq}-1$ and $T:=B+1$. 
Then $r=\frac{n}{pq}<\frac{X}{pq}$, so $r'=CH<Y$, and using(\ref{E2}) we get 
$F>f_L:=\frac{HU+YV}{KY}$; and $C\geq T$ whence $F\leq f_M:=\frac{U+TV}{KT}$. 
So if $F_L:=\lceil f_L\rceil$ and $F_M:=\lfloor f_M\rfloor$, we need integer pairs $(C,F)$ 
such that $F_L\leq F\leq F_M$. 
We easily show that 
\begin{equation*}\begin{split}
f_L(B):=&f_L=\frac{H}{A}+\frac{p}{AB}+\frac{pH(Bp+A)(HB+1)}{AB(X-p-pHB)}\\
\text{and }f_M(B):=&f_M=\frac{H}{A}+\frac{p}{A}\Bigl(\frac{1}{B}+\frac{1}{B+1}\Bigr)+\frac{1}{B(B+1)} 
\end{split}
\end{equation*}

Clearly $f_L<f_M$, but we can get $F_M=\lfloor f_M\rfloor<f_L<f_M<\lceil f_L\rceil=F_L$, 
in which case there are no possible $F$ values. 
Further, as $B\uparrow$, both $f_L(B)\downarrow$ and $f_M(B)\downarrow$, and also 
$F_o=\lfloor\frac{H}{A}\rfloor\leq\frac{H}{A}<f_L\leq F_L$, so if $F_M=F_o$, then we take 
next $A(j)$ (at $X=10^{18}$ Matthew found that this $F_M=F_o$ trigger reduced the program 
time by~20\%). 
It is easily shown that 
\begin{equation*}
 f_M\Bigl(\frac{2p}{\nu}-1\Bigr)=F_o+1+\frac{\nu^2(p+A)}{2Ap(2p-\nu)}>F_o+1>F_o+1-\frac{\nu^2(p-A)}{2Ap(2p+\nu)}
=f_M\Bigl(\frac{2p}{\nu}\Bigr)
\end{equation*}
whence $F_M(\frac{2p}{\nu}):=\lfloor f_M(\frac{2p}{\nu})\rfloor=F_o$, justifying $B<\frac{2p}{\nu}$, 
anticipated in Loop 2.

A further though smaller economy can be achieved by eliminating from consider-ation some or all 
of those $B$ values for which $F_o+1<f_L<f_M<F_o+2$ and so $F_L>F_M$: if then $f_L(\beta)>F_o+1$, 
we can take next $A$; if not, it can be shown that if $\alpha:=\sqrt{\nu^2-\frac{4p^3H^2}{X}}$ 
then $f_L(\frac{2p}{\nu+\alpha})\bumpeq F_o+1$, so we can jump to $B\bumpeq \frac{2p}{\nu+\alpha}$ 
to find $f_L(B)$ just greater than $F_o+1$ and then continue (but this is awkward to program).

We next consider splitting the range (see \S\ref{Sec5c}) and a method of economising on $C$-trials 
which arranges them in an AP, first term $C_o$, say, and common difference $e\in \{1,2,3,6\}$. 
Consider the conditions (a) $2\mid AB$ and (b) $3\mid AB$ and $3\nmid H$. 
If only (a) holds, $C$ is odd and $e=2$. 
If only (b) holds, for $C\equiv -\frac{1}{H}\pmod 3$ 
we have $r=CH+1\equiv 0\pmod 3$, so $C\not\equiv 0\pmod 3$ and $C\not\equiv -\frac{1}{H}\pmod 3$,
leaving only one possible residue, and $e=3$. 
If both (a) and (b) hold, then $e=6$, and if neither, $e=1$. 
Then by eliminating as appropriate for each situation over the range $B+1\leq C\leq B+e$ 
if $2\mid C$, $3\mid C$ or $3\mid r$, we find $C_o$. 
With the notation of \S\ref{Sec5c} and with $x=C$, it seems reasonable to take $k=\frac{1}{e}$ 
and then $\xi=\sqrt{\frac{eU}{K}}\bumpeq\sqrt{eH}$.

So to execute the loop, as described above, we see whether $F_M$ and $F_L$ values permit us to 
take next $A$ (or possibly to jump some $B$'s); and then if $F_L>F_M$ we take next $B$. 
If $B<\sqrt{eH}$ we do $C$-trials until $C\geq\sqrt{eH}$, taking next $B$ if $F<F_L$ while 
$C<\sqrt{eH}$, and then $F$-trials; but if $B\geq\sqrt{eH}$ we do \hbox{$F$-trials} for $F_L\leq F\leq F_M$. 
For $C$-trials $F=\frac{F_T}{F_B}$ where $F_T:=U+VC$ and $F_B:=KC$, so with $V^*=eV$ and 
$K^*=eK$ we start with $C=C_o$ and then do \hbox{$F_T\longrightarrow F_T+V^*$} and 
\hbox{$F_B\longrightarrow F_B+K^*$} 
to find $F$ for next $C$; and similarly for $F$-trials with \hbox{$C=\frac{U}{KF-V}=\frac{U}{E}$} 
we do $E\longrightarrow E+K$ for unit increase in $F$.
Also if $F_L$ gives $E=1$, we take next $E$.

Each $(C,F)$ integer pair then gives $r=CH+1$ which we test for primality, using the standard 
algorithm if $r$ is beyond the bitmap prime data base.

\subsection{Some notes on our implementation of HII}\label{Sec5e}

(i) In the development stage, to test $q$ for primality we used a carefully designed but complicated system
of tracking through a prime database, exploiting the advance of the arrays$\{q_\lambda(i)\}$ by $p\,'$
for each unit increment in $\lambda$, and we also had much less RAM. For these reasons we constructed arrays 
$\{q_\lambda(i)\}$ for each $p$, rather than each $A$, which involved extra complications with
certain loop exits. But then Matthew found that primality testing for $q$
was taking at least 80\% of the time, and constructed the bitmap database,
which at $X=10^{18}$, for example, reduced the program running time
by a factor of at least 5, and was a major contribution to what we 
were able to achieve. Nevertheless we did not revise the array structure,
as we estimated this would have given only a marginal decrease in time.

(ii) In \S\ref{Sec5g} we shall give some running times , so we mention that the Loop 3
$F_M=F_o$ trigger was a late discovery, and right up to $C_3(10^{24})$ our implementation
only used the special case $F_M=0$, which Matthew's later trial showed gives about 
60\% of the 20\% time saving available at $X=10^{18}$.

\subsection{A faster algorithm?}\label{Sec5f}
If $(C^*, F^*)$ is an integer pair, it follows from the theory of PII outlined in~\S~\ref{Sec5c}
(or directly from~(\ref{E2})) that a necessary condition for $(C, F)$
to be an integer pair is $C=C^*+Ku$, and then (\ref{E1}) and~(\ref{E2}) give
\begin{equation}\label{E3}KuF+C^*F-Vu-C^*F^*=0,\end{equation}
an even more discriminating bilinear relation, between $F$ and $u$.
The total number of trials when the split range is used with(\ref{E2}) is approximately
\(2\sqrt{H}-T-\dfrac{H^2}{Y}\), so when this is very large (big $H$, very big $X$),
using~(\ref{E3}) might be worth the cost of finding $(C^*,F^*)$ --- either by the method of PII
for $w$, or simply using~(\ref{E2}) until (and if) such a $(C^*,F^*)$ is encountered.
We did not implement this.

\subsection{Comparison of algorithms for $d=3$}\label{Sec5g}

Ignoring time required for primality testing of $q$ in PI and HII (by
virtue of ``bitmap'') and of $r$ (relatively seldom required and the same for
all four programs), and based on the number of test pairs $(E,D), (p,q), (C,F)$
involved, I deduced that PI, PII and HI are all $O(X^{\frac{2}{3}+o(1)})$; in
the range $10^{12}\leq X\leq 10^{16}$ for all three programs when $X$ was multiplied by 10
the multiplier for the time was close to 4.325 and slowly increasing with $X$, giving
some support to this deduction since $10^{\frac{2}{3}}\bumpeq 4.64$. 
On the same basis I conjecture
but have not succeeded in establishing that HII is $O(X^{\frac{1}{2}+o(1)})$, and over the
range $10^{16}\leq X\leq 10^{18}$ the corresponding time multiplier was 3.013, slightly less than
$10^{\frac{1}{2}}\bumpeq 3.162$.

Here are a few of the many times recorded.
Illustrating the effects of improvements in computer technology, more powerful algorithms and 
the use of a compiled in place of an interpreted language,
Gordon first reached
$C_3(10^{12})=1000$ early in 1998 with an early version of HI on a mid-1980's computer in
about 45~hours; in November 2001 it took Matthew 0.19 seconds actual calculating
``user'' time (1.15 seconds total). Gordon on his new computer (see \S\ref{Sec5a}), with HII
and my $q$ primality testing method (see subsection~\ref{Sec5e}(i)) in mid 1999 took
about 32 seconds for $C_3(10^{12})$ and $35\frac{1}{2}$ hours for $C_3(10^{18})$; in mid 2000 $C_3(10^{18})$
took Matthew $10\frac{3}{4}$~minutes with a slightly improved prime testing method and
compiler optimisation, and finally with this method fully replaced by bitmap it took
just 2~minutes 7.59~seconds. In July 2002 $C_3(10^{24})$ took about 58 hours, with
about 9 minutes for the bitmap.

RAM and time constraints prevented us from going on to $X=10^{25}$.

\subsection{Checking and correction}\label{Sec5h}

Up to $X=10^{18}$, Gordon and I had Richard Pinch's paper \cite{Pinch} and
his Internet results to check against. We achieved agreement up to $10^{17}$,
but at $10^{18}$ we found that his list omitted $n^\dagger=835327 \cdot 893359 \cdot 1117117=833645090806507981$
(for more on $n^\dagger$, see discussion following Theorem~4.1).
Richard told me that $n^\dagger$ inexplicably failed to reach the Internet list,
although his program gave it. He also kindly put me in touch with Carl Pomerance,
who sent me the first preprint of~\cite{Granville} and invited Gordon Davies and
me to attempt the awkward evaluation of the constant $\kappa_3$ (see \cite{Chick1}). Some
months later when Carl asked us for any counts we had beyond $X=10^{18}$, Matthew
had got to $X=10^{20}$, but had not yet done any checks; it later emerged
that a problem in the program was by $X=10^{20}$ unfortunately causing omissions:
the value of 120459 for $C_3(10^{20})$ which we gave to Carl and is published in \cite{Granville}
should be 120625, and the number of imprimitive $C_3N$'s up to $10^{20}$ should be 89854.

Obviously comprehensive checking of HII results for large $X$ with other
known algorithms is not practicable. Soon after successfully programming PI,
Matthew used it for a complete check at $X=10^{19}$; this took about $62\frac{1}{2}$~hours,
checking $q$ for primality by the standard algorithm, and no discrepancy
was found. For final checking he used PII to find the $C_3N$'s corresponding
to every $k^{th}$ $p$-value for \hbox{$X=10^N$} for $(N,k)=$(19,2), (20,10), (21,30), (22,150), (23,1000)
and (24,1000), with initial $p$ values chosen to cut down
repetitions of the same check. No discrep-ancies were found. The last of these
checking runs, at $X=10^{24}$, took PII about 74~hours and HII about $3\frac{1}{2}$~minutes.

We are grateful to Harvey Dubner for collaboration which gave a further
partial check. Let $C_3^\dagger(X):=\#\{n: n$ is a $C_3N$ with $A=1$, and $n\leq X\}$. In~\cite{Dubner}
Dubner finds $C_3^\dagger(10^N)$ up to $N=20$, and suggests that for a ``wide range
of $N$'', $\dfrac{C_3^\dagger(10^N)}{C_3(10^N)}\bumpeq 0.644$. He uses an entirely different algorithm for $C_3^\dagger(X)$, based on relevant $(1,B,C)$ values. 
In correspondence he then took 
$C_3^\dagger(10^N)$ up to $N=24$, obtaining agreement with counts we have extracted
from our discs for $C_3(10^{23})$ and, later, $C_3(10^{24})$. In Table 1 of \S\ref{Sec6} we extend
up to $N=24$ Dubner's Table~2 for $(1,B,C)$ in~\cite{Dubner}.

When finding $C_3(X)$ for $X\geq 10^{18}$, we avoided the danger of rounding
errors wrongly including or excluding a $C_3N$ very close to $X$ by doing a run
to find $C_3(X^*)$ with $X^*=(1+\epsilon)X$ and examining individually any $C_3N$'s
in the range $(1\pm\epsilon)X$, where typically $\epsilon=10^{-3}\text{ or }10^{-4}$ (at $X=10^{24}$ Matthew took $\epsilon=0.1$).

\section{Statistics}\label{Sec6}

In Table~1 we tabulate for $X=10^N$, with $3\leq N\leq 24$, $C_3(X)$ and various other numbers
which we now define. In~\cite{Granville} Granville and Pomerance define {\em primitive} $CN$'s,
and for $C_3N$'s their definition implies that a $C_3N$ is primitive if{}f $H\leq ABC$;
$C_3^*(X):=\#\{n: n\text{ is a primitive }C_3N\text{ and }n\leq X\}$, and our data are consistent
with their conjecture that $\dfrac{C_3^*(X)}{C_3(X)}\rightarrow 0$ as $X\rightarrow \infty$.

Let $\mathcal C:=\{n: n=pqr\text{ is a }C_3N\text{ and }p\equiv q\equiv r\equiv -1 \pmod 4\}$; Rabin showed
in \cite{Rabin} that the probability of any odd composite $n$ passing the strong pseudoprime
test for a randomly chosen base $b$ is less than $\frac{1}{4}$, and that this bound is
approached most closely when $n\in \mathcal C$; and Pinch lists various other properties
of~$\mathcal C$ in~\cite{Pinch}; $\mathcal C(X):=\#\{n: n\in \mathcal C\text{ and }n\leq X\}$

In \S~8 of~\cite{Granville} Granville and Pomerance conjecture that
$C_3(X)\sim \tau_3\dfrac{X^{\frac{1}{3}}}{(\log X)^3}
\sim \dfrac{\tau_3}{27}\displaystyle\int^{X^{\frac{1}{3}}}_2\dfrac{dt}{(\log t)^3}$, where $\tau_3\bumpeq 2100$
 is a constant whose evaluation is discussed in~\cite{Chick1}; they define $\beta$ and $\gamma$ by
$C_3(X)=\beta\dfrac{X^{\frac{1}{3}}}{(\log X)^3}=\dfrac{\gamma}{27}\displaystyle\int^{X^{\frac{1}{3}}}_2\dfrac{dt}{(\log t)^3}$
 and predict that $\beta$ and $\gamma$ eventually
converge to $\tau_3$ from above and below respectively. Our new data are consistent with this, supporting their
cautious comment in \cite{Granville} (but see~\cite{Chick1}, Table~3 and comment).

$C_3^\dagger(X)$ is defined above in \S~\ref{Sec5h}.

\begin{table}[!hbp]
\caption{}\label{Table1}
\begin{tabular}{|rrrrrrrrr|}\hline
$N$ & $C_3(X)$ & $C^*_3(X)$ & $\dfrac{C^*_3(X)}{C_3(X)}$ & $\mathcal C(X)$ & $\beta$ & $\gamma $
& $C^\dagger_3(X)$ & $\dfrac{C^\dagger_3(X)}{C_3(X)}$  \\\hline
3       & 1     & 1     		& 1     & 0     & 32.96 & 9.092 & 1     & 1 	\\
4	& 7	& 7			& 1	& 1	& 253.9	& 53.13	& 6	&0.8571	\\
5	& 12	& 12			& 1	& 1	& 394.5	& 78.07	& 11	&0.9167	\\
6	& 23	& 19			& 0.826	& 1	& 606.5	& 128.1	& 18	&0.7826	\\
7	& 47	& 36			& 0.766	& 4	& 913.5	& 220.2	& 36	&0.7660	\\
8	& 84	& 59			& 0.702	& 8	& 1131	& 321.9	& 59	&0.7024	\\
9	& 172	& 113			& 0.657	& 15	& 1531	& 519.9	& 122 	&0.7093	\\
10	& 335	& 208			& 0.621	& 29	& 1898	& 761.3	& 227	&0.6776	\\
11	& 590	& 338			& 0.573	& 50	& 2065	& 961.5	& 403	&0.6831	\\
12	& 1000	& 529			& 0.529	& 79	& 2110	& 1113	& 680	&0.68\ \ \ \ 	\\
13	& 1858	& 930			& 0.501	& 153	& 2313	& 1349	& 1220	&0.6566	\\
14	& 3284	& 1550			& 0.472	& 271	& 2370	& 1496	& 2104	&0.6407	\\
15	& 6083	& 2621			& 0.431	& 487	& 2506	& 1680	& 3911	&0.6429	\\
16	& 10816	& 4201			& 0.388	& 868	& 2510	& 1763	& 6948	&0.6424	\\
17	& 19539	& 6814			& 0.349	& 1569	& 2525	& 1839	& 12599	&0.6448	\\
18	& 35586	& 11190			& 0.314	& 2837	& 2534	& 1899	& 22920	&0.6441	\\
19	& 65309	& 18432			& 0.282	& 5158	& 2538	& 1947	& 41997	&0.6431	\\
20	& 120625	& 30771		& 0.255	& 9443	& 2538	& 1984	& 77413	&0.6418	\\
21	& 224763	& 51432		& 0.229	& 17316	& 2541	& 2019	&144300	&0.6420	\\
22	& 420658	& 85921		& 0.204	& 32351	& 2538	& 2047	&270295	&0.6426	\\
23	& 790885	& 143620 	& 0.182	& 61130	& 2531	& 2067	&508780	&0.6433	\\
24	& 1494738	& 241562	& 0.162	&115606	& 2523	& 2081	&961392	&0.6432	\\\hline
\end{tabular}
\end{table}

\newpage
From Theorem~\ref{Th3_3} we have $n\leq N_3(P):=\frac{1}{2}(P^6+2P^5-P^4-P^3+2P^2-P).$

For any odd prime $p$ we define \(\chi(p\,):=\#\{n: n \text{ is a }C_3N \text{ and } n=p\,qr\}\)
and \(T(x):=\sum_{p\leq x}\chi(p\,)\).
Our list of $C_3N$'s up to $10^{24}$ enables us to find $\chi(p)$ and $T(p)$ up to $p=11213$,
since we have $N_3(11213)<10^{24}<N_3(11239)$. 
In Table~2 we tabulate $p\,, \chi(p\,)$ and $T(p\,)$ up to $p=211$, and in Table~3 $p, T(p\,)$ 
for $\pi(p\,)$ in intervals of 50 or 240 up to $\pi(11213)=1357$.

$\chi(p\,)=0$ for \(p= 11, 197, 1223, 1487, 4007, 4547, 7823, 9833, 9839\) and 10259,
and $\chi(p\,)=1$ for 51 values of $p$.
$\chi(211)=17$ is the greatest value of $\chi(p\,)$ until $p=1171$;
with $\chi(p)\geq 22$ we have \(\chi(p)=22\text{ for }p=1171, 7481, 8521, 8647\) and 10711,
\(\chi(9241)=24, \chi(10837)=25\text{ and } \chi(2221)=29.\)

At first in Tables~2 and~3 $T(p\,)$ keeps remarkably close to $p$ before going ahead for a bit,
but then $p$ gradually overhauls $T(p\,)$ and seems to be slowly pulling away.
Clearly $C_3(X)\leq T(p_M),$ and a plausible heuristic argument that
\(T(p_M) \leq O(X^{\frac{1}{3}+o(1)})\)
can be based on the loops of algorithm HII, ignoring the primality requirement
on $p, q, r$. The best upper bound for $C_3(X)$ which has so far been proved is
$O(X^{\frac{5}{14}+o(1)})$, by Balasubramanian and Nagaraj in~\cite{Balas}.

\begin{table}[!htp]\label{Table2}
\caption{Number and cumulative total of $C_3N$'s with first prime $p$}
\begin{tabular}{|l|rrrrrrrrrrrr|}\hline
$p$ 		& 3	& 5	& 7	& 11	& 13	& 17	& 19	& 23	& 29	& 31	& 37	& 41	\\
$\chi(p\,)$	& 1	& 3	& 6	& 0	& 5	& 2	& 2	& 1  	& 2	& 7	& 5	& 7	\\
$T(p\,)$ 	& 1	& 4	& 10 	& 10	& 15	& 17	& 19	& 20	& 22	& 29	& 34	& 41	 \\
\hline
$p$ 		& 43	& 47    & 53 	& 59	& 61	& 67	& 71	& 73	& 79	& 83	& 89	& 97	 \\
$\chi(p\,)$	& 11	& 3	& 3 	& 1	& 10	& 3	& 7	& 4	& 1	& 2	& 5	& 6	 \\
$T(p\,)$ 	& 52	&55 	& 58	& 59	& 69	& 72	& 79	& 83	& 84	& 86	& 91	& 97	 \\
\hline
$p$		& 101	& 103	& 107	& 109	& 113	& 127	& 131	& 137	&139   	& 149 	& 151	& 157	 \\
$\chi(p\,)$	& 2	& 5	& 3	& 10	& 5	& 5	& 11 	& 4   	& 6	& 2	& 9	& 11	 \\
$T(p\,)$	& 99	& 104	& 107	& 117	& 122	& 127	& 138 	& 142 	& 148	& 150	& 159	& 170	 \\
\hline
$p$		& 163	& 167	& 173	& 179	& 181  	& 191	& 193	& 197	& 199 	& 211	&	&	\\
$\chi(p\,)$	& 7	& 2	& 3	& 4	& 11  	& 6	& 10	& 0	& 7	& 17 	&	&	\\
$T(p\,)$	& 177	& 179	& 182	& 186	& 197 	& 203	& 213	& 213	& 220	& 237	&	&	\\
\hline
\end{tabular}
\end{table}

\begin{table}[!hbp]\label{Table3}
\caption{Cumulative total in Table 2 extended}
\begin{tabular}{|l|rrrrrrrr|rrrr|}\hline
$\pi(p\,)$	& 47	& 97	& 147	& 197	& 247	& 297	& 347	& 397	& 637	& 877	& 1117	& 1357 \\
$p$		& 211	& 509	& 853	& 1201	& 1567	& 1951	& 2341	& 2719	& 4723	& 6823	& 8999	& 11213	 \\
$T(p\,)$	& 237	& 565	& 896	& 1235	& 1556	& 1906	& 2299	& 2651	& 4347	& 6110	& 7945	& 9608 \\
\hline
\end{tabular}
\end{table}

From (\ref{D19}) and Theorem~\ref{Th4_9}(d) we have
\(n\leq N(H)=\frac{1}{2}(H^{10}+4H^9+14H^8+30H^7+53H^6+69H^5+71H^4+55H^3+31H^2+12H+2)\).

Since \(N(268)<10^{24}<N(270)\),
we can in a similar way use 
our list of $C_3N$'s up to $10^{24}$ to count all the $C_3N$'s for each $H$ up to $H=268$.
Let \linebreak
$\zeta(H):=\#\{n: n \text{ is a }C_3N \text{ with } \gcd[p\,',q\,']=H\}$
and \(Z(x):=\sum_{H\leq x}\zeta(H)\).
We find $\zeta(H)=0$ for $H=$ 68, 76, 160, 176, 188, 196 and 218,
$\zeta(H)=1$\linebreak for $H=$ 98, 104, 134, 164, 184, 202, 212, 232, 244 and 248;
and the largest values are $\zeta(210)=19$, $\zeta(H)=18$  
for $H= 30, 60, 102 \text{ and } 156$, $\zeta(150)=16$ and $\zeta(198)=13$.
Table~4 shows the growth of~$Z(H)$.
\begin{table}
\caption{Cumulative total of $C_3N$'s with $\gcd(p\,',q\,')\leq H$}
\begin{tabular}{|l|rrrrrrrrrrrrrr|}\hline
$H$    & 20 & 40 & 60  &  80 & 100 & 120 & 140 & 160 & 180 & 200 & 220 & 240 & 260 & 268 \\
$Z(H)$ & 44 & 98 & 162 & 204 & 263 & 334 & 390 & 450 & 491 & 531 & 578 & 646 & 687 & 705 \\
\hline
\end{tabular}
\end{table}

\newpage
Table~5 shows \(\#\{n: n \text{ is a } C_3N, n \equiv c \pmod m \text{ and } n\leq 10^N\}\)
for various $m$, $c$ and $N$.

\begin{table}[!hbp]\label{Table5}
\caption{Cumulative totals of $C_3N$'s up to $10^N$ satisfying $n \equiv c \pmod m$}
\begin{tabular}{|rrr|rrrrrrrrr|}\hline
	&	& N	& 7	& 9	& 11	& 13	& 15	& 17	& 19	& 21	& 23	\\
m	& c	&  	&	&	&	&	&	&	&	&	&	\\
\hline
5	& 1	&	& 35	& 133	& 457	& 1405	& 4611	& 14716	& 49030	&169157	&595168	\\
	& 2	&	& 1	& 6	& 40	& 133	& 455	& 1522	& 5151	& 17479	& 61711	\\
        & 3	&	& 5	& 11	& 41	& 138	& 434	& 1421	& 4726	& 16108	& 56953	\\
        & 4	&	& 3	& 19	& 49	& 179	& 580	& 1877	& 6399	& 22016	& 77050	\\
\hline
7	& 1	&	& 22	& 102	& 339	& 1078	& 3472	& 11029	& 36668	&125774	&443797	\\
        & 2	&	& 4	& 9	& 36	& 136	& 499	& 1660	& 5590	& 19280	& 68227	\\
        & 3	&	& 4	& 18	& 54	& 171	& 501	& 1636	& 5645	& 19551	& 68150	\\
        & 4 	&	& 3	& 12	& 55	& 162	& 544	& 1766	& 6057	& 20990	& 73529	\\
	& 5	&	& 4	& 10	& 46	& 133	& 494	& 1666	& 5379	& 18752	& 65616	\\
        & 6	&	& 4 	& 15	& 54	& 172	& 567	& 1776	& 5964	& 20410	& 71560	\\
\hline
11      & 1	&	& 13	& 48	& 183	& 591	& 2063	& 6678	& 22417	& 77368	&272654	\\
        & 2	&	& 3	& 18	& 54	& 161	& 471	& 1499	& 4965	& 17230	& 60546	\\
        & 3	&	& 2	& 13	& 43	& 134	& 432	& 1367	& 4729	& 16599	& 58510	\\
        & 4	&	& 4	& 13	& 40	& 142	& 421	& 1362	& 4598	& 15971	& 55563	\\
        & 5	&	& 4	& 13	& 47	& 151	& 435	& 1411	& 4756	& 16204	& 57647	\\
        & 6	&	& 6	& 15	& 49	& 151	& 478	& 1515	& 4944	& 16670	& 58038	\\
        & 7	&	& 2	& 10	& 34	& 127	& 443	& 1378	& 4711	& 16318	& 57489	\\
        & 8	&	& 5	& 14	& 46	& 130	& 455	& 1511	& 4869	& 16543	& 57988	\\
        & 9	&	& 4	& 12	& 48	& 115	& 420	& 1431	& 4748	& 16233	& 57508	\\
        & 10	&	& 3	& 15	& 45	& 155	& 464	& 1386	& 4571	& 15626	& 54941	\\
\hline
12      &  1	&	& 38	& 145	& 516	& 1632	& 5353	& 17221	& 57694	&199002	&700227	\\
        & 5	&	& 4	& 11	& 23	& 72	& 242	& 748	& 2456	& 8444	& 29527	\\
        & 7	&	& 4	& 15	& 50	& 153	& 478	& 1517	& 4994	& 16766	& 59215	\\
        & 11	&	& 0	& 0	& 0	& 0	& 9	& 52	& 164	& 550	& 1915	\\
\hline
\end{tabular}
\end{table}

In \cite{Pinch} Pinch describes and searches for certain special types of $CN$ discussed by
other authors, including {\em strong Fibonacci pseudoprimes} (of which he finds just one
up to $10^{18}$, with $d=8$). These special $CN$'s all have the property that $p\,_i+1$ divides
$n\pm 1$ for $1\leq i\leq d$, and for $d=3$ I have proved that no such numbers exist: see~\cite{Prob}.

\section{Acknowledgments}\label{Sec7}
I have already acknowledged the huge contributions of Gordon Davies and \break Matthew Williams.
Their collaboration has been essential and most rewarding, and I am \hbox{extremely} grateful.
As mentioned in \S~\ref{Sec5c}, HI was originally devised
to find $C_3N$'s which were also Perrin pseudoprimes in an earlier (unpublished)
investigation and it was Gordon who originally suggested using it to pursue $C_3(X)$.
I am also deeply indebted to my friend Ian Williams (Haileybury physics teacher and \hbox{father} of Matthew)
for extracting from discs listing $C_3N$'s up to $10^{23}$, and then $10^{24}$,
supplied by Matthew, the data for Tables~2, 3, 4,~5
and $\mathcal{C}(X) \text {and } C^\dagger_3(X)$ in Table~1; 
for doing the $K_3N$ or $C_3N$ computer checks and searches required for 
and associated with Theorems~\ref{Th3_1}, \ref{Th3_3} and \ref{Th4_8}; 
and also for undertaking the massive task of converting my manuscript into $\mathcal{AMS}$-\LaTeX .
I also thank Richard Pinch and Carl Pomerance for their encouragement and the stimulation
of their work.

\bibliographystyle{amsplain}

\end{document}